\documentclass[11pt]{article}
\usepackage{latexsym,amsmath,amssymb,amsthm,amsfonts,graphicx, color}
\usepackage{graphicx}
\numberwithin{equation}{section}
\numberwithin{figure}{section}

\usepackage[bookmarks,bookmarksnumbered]{hyperref}
\newtheorem{theorem}{Theorem}[section]
\newtheorem{lemma}[theorem]{Lemma}
\newtheorem{proposition}[theorem]{Proposition}

\numberwithin{equation}{section}
\numberwithin{equation}{section}
\Large
\begin{document}

\title{Exponential Convergence to the Maxwell Distribution For Some Class of Boltzmann Equations}
\author{J\"urg Fr\"ohlich\footnote{juerg@itp.phys.ethz.ch}, Zhou Gang\footnote{zhougang@itp.phys.ethz.ch}}
\maketitle
\setlength{\leftmargin}{.1in}
\setlength{\rightmargin}{.1in}
\normalsize \vskip.1in \setcounter{page}{1}
\setlength{\leftmargin}{.1in} \setlength{\rightmargin}{.1in}
\centerline{Institute for Theoretical Physics, ETH Zurich, CH-8093, Z\"urich, Switzerland}
\section*{Abstract}
We consider a class of nonlinear Boltzmann equations describing return to thermal equilibrium in a gas of colliding particles suspended in a thermal medium. We study solutions in the space $L^{1}(\Gamma^{(1)},d\lambda),$ where $\Gamma^{(1)}=\mathbb{R}^{3}\times \mathbb{T}^3$ is the one-particle phase space and $d\lambda= d^3 v d^3 x$ is the Liouville measure on $\Gamma^{(1)}.$ Special solutions of these equations, called ``Maxwellians," are spatially homogenous static Maxwell velocity distributions at the temperature of the medium. We prove that, for dilute gases, the solutions corresponding to smooth initial conditions in a weighted $L^{1}$-space converge to a Maxwellian in $L^{1}(\Gamma^{(1)},d\lambda),$ exponentially fast in time.
\tableofcontents
\section{Physics Background}
In this paper we study the phenomenon of ``return to equilibrium" for a gas of particles suspended in a thermal medium, in the limit where the range, $D$, of two-body forces between pairs of particles tends to $0$, while $\rho D^2$ is kept constant, with $\rho$ the density of the gas (Boltzmann-Grad limit). We assume that the one-particle phase space, $\Gamma^{(1)},$ is given by
\begin{equation}
\Gamma^{(1)}=\mathbb{R}^3\times \mathbb{T}^3,
\end{equation} where $\mathbb{T}^3=\mathbb{R}^3/L\mathbb{Z}^3$, $L\not=0$, is configuration space (a three-dimensional, flat torus of diameter L), and $\mathbb{R}^3$ is velocity space. Boltzmann's hypothesis of ``molecular chaos" is the assumption that the $n-$ particle correlation functions describing the initial state of the gas at time $t=0$ are given by an $n-$fold product
$$\prod_{j=1}^{n}g_0(v_j,x_j),\ (v_j,x_j)\in \Gamma^{(i)},\ j=1,\cdots,n,\ n=1,2,3,\cdots,$$ of a one-particle density, $g_0(v,x)$, on $\Gamma^{(1)}.$ One expects that, in the Boltzmann-Grad limit, molecular chaos propagates from the initial state to the state of the gas at an arbitrary later time, i.e., that the $n-$particle correlation functions at time $t>0$ are given by
$$\prod_{j=1}^{n}g_{t}(v_j,x_j)$$ where $g_{t}$ is the solution of a Boltzmann equation with initial condition given by $g_{t=0}=g_0;$ (see ~\cite{Lanford1975} for important results in this direction).

In this paper, we assume that every particle in the gas interacts with a memory-less thermal medium of temperature $T>0.$ Physically, this assumption is quite natural. It appears to us, however, that the corresponding mathematical problems have not received the attention they deserve. Assuming first that the gas consists of a single particle, we expect that the time evolution of its state in the van Hove limit, where the strength, $\lambda$, of the interaction of the particle with the medium tends to $0$, but time is scaled by a factor $\lambda^{-2}$, is given by {\it{a linear Boltzmann equation}} of the form
\begin{equation}\label{eq:LinBolt}
\partial_{t}g+v\cdot \nabla_{x} g=-L g+G[g],
\end{equation} where
\begin{equation}\label{eq:loss}
(Lg)(v,x):=\nu_0(v) g(v,x),
\end{equation} with $\nu_0(v)=\int_{\mathbb{R}^3} r_0(u,v)\ d^3 u$, is a ``loss term", and
\begin{equation}\label{eq:gain}
G[g](v,x):=\int_{\mathbb{R}^3} r_0(v,u)\ g(u,x)\ d^3 u
\end{equation} is a ``gain term". The kernel $r_0(u,v)$ is assumed to obey ``{\it{detailed balance}}", i.e.,
\begin{equation}\label{eq:DetailBal}
r_0(u,v)=r_0(v,u)e^{\frac{\beta m}{2}(|v|^2-|u|^2)},
\end{equation} where $\beta=(k_{B}T)^{-1}$ denotes the inverse temperature, $m$ is the particle mass, and $\frac{m}{2}|v|^2$ is the kinetic energy of a non-relativistic particle of mass $m$ and velocity $v.$

The equation $\partial_{t}g+v\cdot \nabla_{x} g=0$ describes an inertial motion of a particle with velocity $v$ distributed over $\mathbb{T}^3$ according to $g_{t}(v,x)$. The right hand side of Equation ~\eqref{eq:LinBolt} describes the effects on the motion of the particle of its interactions with a thermal medium at temperature $T>0,$ in the van Hove limit; (see, e.g. ~\cite{Davies1976,ErdhosYau}). Next, we consider a gas of $N \simeq \rho L^3$ particles interacting with each other and with the medium. (Here $\rho$ is the density of the gas and $L^3$ the volume of $\mathbb{T}^3$). We assume that the medium has no memory (i.e., that it equilibrates arbitrarily rapidly after each interaction with a particle) and that the interactions between the particles in the gas are given by a two-body potential of short range (possibly induced by exchange of modes of the thermal medium). Let $\frac{d\sigma}{d\omega}$ denote the differential cross section for scattering between two particles in the given two-body potential. Let $u$ and $v$ be the velocities of two incoming particles and $u'$, $v'$ their outgoing velocities after an elastic collision process. By energy-momentum conservation,
\begin{equation}\label{eq:InOut}
u'=u-[(u-v)\cdot \omega]\omega,\ v'=v+[(u-v)\cdot \omega]\omega,
\end{equation} where $\omega$ is a unit vector. We define
\begin{equation}\label{eq:collision}
Q(g,g)(v,x):=\int [g(v',x)g(u',x)-g(v,x) g(u,x)]\ |u-v| \frac{d\sigma}{d\omega} d^2\omega d^3 u.
\end{equation} Then the Boltzmann equation for the time evolution of the one-particle density, $g_t(v,x)$, of a gas of $N$ interacting particles coupled to the thermal medium takes the following form:
\begin{equation}\label{eq:nonBo}
\partial_{t}g(v,x)+v\cdot\nabla_{x} g(v,x)=-\nu_0(v) g(v,x)+\int_{\mathbb{R}^3} r_0 (v,u) g (u,x)\ d^3 u+\kappa Q(g,g)(v,x),
\end{equation} where $\nu_0$ is as in ~\eqref{eq:loss} and $r_0$ as in ~\eqref{eq:DetailBal}, $Q(g,g)$ is given by ~\eqref{eq:collision}, and $\kappa$ is the number of moles of the gas. We are interested in solutions, $g_{t}(v,x)$, of ~\eqref{eq:nonBo} with the properties that $g_{t}(v,x)\geq 0$ and $\int_{\Gamma^{(1)}}g_{t}(v,x)\  d^3 vd^3 x=1.$

Under ``reasonable" assumptions (to be specified below) on the kernel $r_0$ and the cross section $\frac{d\sigma}{d\omega}$ (as a function of $\omega$ and of $u,\ v$), a local existence- and uniqueness theorem for smooth solution of ~\eqref{eq:nonBo} corresponding to smooth initial conditions $g_{t=0}(v,x)=g_0(v,x)\geq 0,$ with $\int_{\Gamma^{(1)}}g_0(v,x)\ d^3 v d^3 x=1,$ has been established; (see, e.g., ~\cite{Wenn1994,LuX1998,MouVi}). As a consequence, one may show that, for all times $t>0$ at which $g_{t}$ is known to exist,
\begin{itemize}
\item[(A)] $g_{t}(v,x)\geq 0$ whenever $g_0(v,x)\geq 0;$
\item[(B)]
$
\int_{\Gamma^{(1)}} g_{t}(v,x) \ d^3 vd^3 x=\int_{\Gamma^{(1)}} g_{0}(v,x)\ d^3 vd^3 x=1;
$
\item[(C)] $g(v,x)=C e^{-\frac{\beta m}{2}|v|^2}$ is a static (time-independent) solution of ~\eqref{eq:nonBo}, for a positive constant $C$. These static solutions are henceforth called ``Maxwellians".
\end{itemize}

The purpose of this paper is to prove {\it{asymptotic stability of Maxwellians.}} Our main result says that, under suitable decay- and smoothness assumptions on the initial condition $g_{0}(v,x)$, with $\int_{\Gamma^{(1)}} g_0(v,x)\ d^3 v d^3 x=1,$ and for sufficiently small values of the mole number, $\kappa$, of the gas, a global solution, $g(v,x),$ satisfying (A) and (B) exists and converges to the Maxwellian $Ce^{-\frac{\beta m}{2}|v|^2}$ (independent of $x$), with $C=L^{-3} (\frac{\beta m}{2\pi})^{\frac{3}{2}},$ {\it{exponentially fast in time}}. This result describes the phenomenon of ``exponential return to equilibrium" in a gas of particles suspended in a thermal medium. The velocity distribution of the particles inherits the temperature of the thermal medium thanks to the ``detailed balance condition" ~\eqref{eq:DetailBal}. A precise formulation of our result is presented in Theorem ~\ref{THM:MainTHM}, below.

In the literature, one finds many results on the asymptotic stability of Maxwellians for the Boltzmann equation with $r_0\equiv 0$ and $\kappa$ arbitrary. One circle of results concerns the spatially homogeneous case, where $g(v,x)$ is independent of the position $x$. This direction of research has been pioneered by T. Carleman in ~\cite{Carleman1960}. Further results can be found in ~\cite{MR0156656,Bodmer1973,CerIllPulv,Glassey1996,Mouhot2006}. Another circle of results concerns the Boltzmann equation on an exponentially weighted $L^2$ space; see, e.g. ~\cite{UKai1974,YGuo2002,YGuo2003,MR2629879}. The advantage of working in such spaces is that spectral theory on Hilbert space can be used.

From the point of view of physics, however, the space $L^{1}(\Gamma^{(1)},d\lambda)$, where $d\lambda$ is the Liouville measure on $\Gamma^{(1)}$, is the natural choice for a study of the Boltzmann equation ~\eqref{eq:nonBo}, because the function $g_{t}(v,x)$ has the interpretation of a probability density on $\Gamma^{(1)}.$ In this context, the existence of weak global solutions has been established in ~\cite{DiPernaLion1989}. In ~\cite{MR2116276, Mouhot2010}, the asymptotic stability of Maxwellians, for general initial conditions, has been studied under the assumption that global smooth solutions exist. In the spatially homogeneous case, such results appear, e.g. in ~\cite{Ark1988,Wenn1993, MR1233644, Boby1997, MisWenn1999, MR1264851,BobyGamPan04,Mouhot2006}.

When the nonlinearity $Q$ in ~\eqref{eq:nonBo} is absent, the equation is known as neutron transport equation. In certain settings, the spectrum of the propagator generated by the linear operator, on spaces $L^{p},\ 1\leq p\leq \infty$, has been studied in ~\cite{MR2148153,MR2216092}.

In this paper, we study the simpler problem of Boltzmann equations describing a gas of particles interacting with a thermal medium that tunes the temperature of the asymptotic Maxwell velocity distribution. The simplifications in our analysis, as compared to the usual Boltzmann equation {\it{without}} thermal medium, arise from the presence of the linear gain- and loss terms on the right hand sides of ~\eqref{eq:LinBolt} and ~\eqref{eq:nonBo}; (see ~\eqref{eq:loss}, ~\eqref{eq:gain}). The behavior of solutions of ~\eqref{eq:LinBolt}, for large times, is well understood. One may then view the nonlinearity, $\kappa Q(g,g)$, in ~\eqref{eq:nonBo} as a perturbation. More precisely, we propose to linearize solutions of ~\eqref{eq:nonBo} around the Maxwellian found by solving ~\eqref{eq:LinBolt}, as time $t\rightarrow \infty.$ We must then study the properties of a certain linear operator $L$ defined in Equation ~\eqref{eq:ALinear}, below. An important step in our analysis consists in proving an appropriate decay estimate for the linear evolution given by $e^{-tL}(1-P_0)$, where $P_0$ is the Riesz projection onto the eigenspace of $L$ corresponding to the eigenvalue $0$, which is spanned by the Maxwellian. What complicates this problem is that, for physically relevant choices of $r_0$ and cross sections $\frac{d\sigma}{d\omega}$, the spectrum of the operator $L$ occupies the entire right half of the complex plane, except for a strip of strictly positive width around the imaginary axis that only contains the eigenvalue $0$; see Figure ~\ref{fig:FigureExample}, below. Rewriting $e^{-tL}(1-P_0)$ in terms of the resolvent, $(L-z)^{-1},$ of $L$,
\begin{equation}\label{eq:spectralTHMini}
e^{-tL}(1-P_0)=-\frac{1}{2\pi i}\oint_{\Gamma} e^{-tz}(L-z)^{-1}\ dz,
\end{equation} (see, e.g., ~\cite{RSI}), where the integration contour $\Gamma$ encircles the spectrum of $L$, except for the eigenvalue $0$, we encounter the problem of proving strong convergence of the integral on the right hand side of ~\eqref{eq:spectralTHMini} on $L^{1}$. This problem is solved in Section ~\ref{sec:propagatorEst}. We expect that an extension of our techniques can be used to prove a conjecture in ~\cite{RezVillani2008} concerning the exponential convergence of solutions of the Boltzmann equation to a Maxwell distribution. For the results in this direction, see \cite{Mouhot2010} for a constructive proof, and \cite{MR900501} for a non-constructive proof.

Our paper is organized as follows. The main hypothesis on the kernel $r_0(u,v)$ and the cross section $\frac{d\sigma}{d\omega}$ and the main result, Theorem ~\ref{THM:MainTHM}, of our analysis are described in Section ~\ref{sec:MainTHM}. In Section ~\ref{sec:Formulation}, the Boltzmann equation ~\eqref{eq:nonBo} is rewritten in a more convenient form; see Equation ~\eqref{eq:ALinear}. The local wellposedness of Equation ~\eqref{eq:ALinear} is proven in Section ~\ref{sec:LocalWell}. In Section ~\ref{sec:propagatorEst}, a decay estimate on the propagator, $e^{-tL}(1-P_0)$, is established. This represents the technically most demanding part of our analysis. The proof of our main result is completed in Section ~\ref{sec:ProofMainTHM}. Three appendices contain some technical details.

\section*{Acknowledgments}
The second author wishes to thank C. Mouhot for pointing out many references.

\section{Explicit Form of the Equation and Main Theorem}\label{sec:MainTHM}
We use the notation $g_{t}(v,x)=:g(v,x,t)$, $(v,x)\in \mathbb{R}^3\times \mathbb{T}^3,$ $t\in \mathbb{R},$ and consider the equation (see ~\eqref{eq:nonBo})
\begin{equation}\label{eq:NLBL1}
\partial_{t} g+v\cdot\nabla_{x} g=-\nu_0 g+\int_{\mathbb{R}^{3}} r_{0}(v,u) g(u,\cdot)\ d^3 u+\kappa Q( g, g)
\end{equation} with initial condition
$$g(v,x,0)= g_0(v,x)\geq 0, \ \ x\in \mathbb{R}^{3}/(2\pi \mathbb{Z})^3\ (\text{i.e.,} \ L=2\pi).$$
The different terms on the right hand side are chosen as follows.
\begin{itemize}
\item[(1)]
The function $\nu_0:\mathbb{R}^3\rightarrow \mathbb{R}^{+}$ is defined by
\begin{equation}\label{eq:difNu0}
\nu_{0}(v):=\int_{\mathbb{R}^{3}} r_{0}(u,v)\ d^3 u.
\end{equation}
\item[(2)] The function $r_0 :\mathbb{R}^3\times \mathbb{R}^{3}\rightarrow \mathbb{R}^{+}$ must satisfy the detailed balance condition ~\eqref{eq:DetailBal}. In the following, we set $\beta m=2.$ The example we have in mind is given by
\begin{equation}\label{eq:R0exam}
r_{0}(u,v):=e^{-|u|^2}(1+|u-v|^2)^{\frac{1}{2}}.
 \end{equation} More generally, we require the following conditions on $r_0$: (a) There exists a positive constant $C>0$ such that $$\frac{1}{C}e^{-|u|^2}(1+|u-v|^2)^{\frac{1}{2}}\leq r_0(u,v)\leq Ce^{-|u|^2}(1+|u-v|^2)^{\frac{1}{2}}.$$ (b) There exists a constant $C_2>0$ such that the derivatives of $r_0$ satisfies the condition $$|\partial_{u}^{k}\partial_{v}^{l}r_{0}(u,v)|\leq C_2 e^{-\frac{1}{2}|u|^2}(1+|u-v|^2)^{\frac{1-l}{2}}$$ for $k+l\leq 1.$
\item[(3)] The constant $\kappa$ is positive and small.
\item[(4)] The nonlinearity $Q( g, g)$ is chosen to correspond to a hard-sphere potential:
\begin{equation}\label{eq:difColli}
Q(g,g)(v,x):=\int_{\mathbb{R}^3\times \mathbb{S}^2} |(u-v)\cdot \omega|[g(u',x)g(v',x)-g(u,x)g(v,x)]\ d^3 u\ d^2 \omega,
\end{equation}
where $u', v'\in \mathbb{R}^{3}$ are given by $u':=u-[(u-v)\cdot \omega]\omega,$ $v':=v+[(u-v)\cdot\omega]\omega$, see Equations ~\eqref{eq:InOut}, ~\eqref{eq:collision}.
\end{itemize}

We define a function $M:\mathbb{R}^{3}\rightarrow \mathbb{R}^{+}$ by
$$M(v):=e^{-|v|^2},$$ and a constant $C_{\infty}$ by
$$C_{\infty}:= \frac{\int_{\mathbb{R}^3\times \mathbb{T}^3} g_{0}(v,x)\ d^3 v d^3 x}{(2\pi)^3\int_{\mathbb{R}^3} M(v) \ d^3 v}$$ and we set $\langle v\rangle:=\sqrt{1+|v|^2}.$

The main result of this paper is the following theorem:
\begin{theorem}\label{THM:MainTHM}
We assume that $$\sum_{|\alpha|\leq 8}\|\langle v\rangle^{m}\partial_{x}^{\alpha} g_0\|_{L^{1}(\mathbb{R}^{3}\times \mathbb{T}^{3})}\leq C$$ for a constant $C<\infty$ and for some sufficiently large $m>0,$ and that the constant $\kappa>0$ in ~\eqref{eq:NLBL1} is sufficiently small. Then there exist positive constants $C_0$, $C_{1}$ such that
\begin{equation}\label{eq:expon}
\| g(\cdot,t)-C_{\infty} M\|_{L^{1}(\mathbb{R}^{3}\times \mathbb{T}^{3})}\leq C_{1} e^{-C_0 t}.
\end{equation}
\end{theorem}
This theorem will be proven in Section ~\ref{sec:ProofMainTHM}.

Concerning the choices of $r_0$ and the collision term $Q$ in ~\eqref{eq:NLBL1}, we make the following remarks.
\begin{itemize}
\item[(A)]
We expect that our results hold under more general assumptions.
For example, if $r_{0}(u,v)=e^{-|u|^2}h(|u-v|),$ where $h$ is a strictly positive, smooth bounded function, and if the collision term $|u-v|\frac{d\sigma}{d\omega}$ (see ~\eqref{eq:collision}), is bounded then it becomes quite easy to prove a result similar to Theorem ~\ref{THM:MainTHM}.
\item[(B)]
If the collision term is unbounded then it simplifies life to impose the condition that $r_0$ is unbounded too. For technical details we refer to Equations ~\eqref{eq:integrationP} and ~\eqref{eq:inteByParts} and the remarks thereafter.
\item[(C)] In our spectral analysis of the linear operator $L$, to be defined in ~\eqref{eq:difL} below, the unboundedness of $\nu_0$ in ~\eqref{eq:NLBL1}, which is defined in terms of $r_0,$ is used. We believe that this is not essential, although it makes proofs simpler. In fact, by results proven in ~\cite{Ark1988,Mouhot2006}, one can generalize our results in Lemma ~\ref{LM:roughEst}.
\end{itemize}

\section{Reformulation of the Boltzmann Equation ~\eqref{eq:NLBL1}}\label{sec:Formulation}
To facilitate later analysis we reformulate equation ~\eqref{eq:NLBL1} in a more convenient form.
We define a function $f:\ \mathbb{R}^{3}\times\mathbb{T}^{3}\times \mathbb{R}^{+}\rightarrow \mathbb{R}$ by
\begin{equation}\label{eq:difF}
f(v,x,t):= g(v,x,t)-C_{\infty}M(v),
\end{equation} with the constant $C_{\infty}$ and function $M$ defined before Theorem ~\ref{THM:MainTHM}.
From ~\eqref{eq:NLBL1} we derive an equation for $f,$
\begin{equation}\label{eq:ALinear}
 \partial_{t}f=-Lf+\kappa Q(f,f).
\end{equation}
Here the nonlinear term $Q(f,f)$ is defined in ~\eqref{eq:difColli}, the linear operator $L$ is defined by
\begin{equation}\label{eq:difL}
L:=v\cdot\nabla_{x}+L_0+ C_{\infty}\kappa L_{1}.
\end{equation} Here $L_0$ and $L_1$ are defines as
\begin{itemize}
\item[(1)]
$$L_{0}:=\nu_{0}-K_{0},$$ where $\nu_0$ is defined in ~\eqref{eq:difNu0}, and for any function $f$, $K_0$ is defined as
\begin{equation}\label{eq:difK0}
K_{0}(f)(v):=\int_{\mathbb{R}^{3}} r_{0}(v,u) f(u)\ d^3 u.
\end{equation}
\item[(2)]
\begin{equation}
L_{1}f:=\nu_1(v)f+K_{1}(f)
\end{equation}
where $\nu_1$ is the multiplication operator defined by $$\nu_1(v):=\int_{\mathbb{R}^3\times \mathbb{S}^2}|(u-v)\cdot\omega| M(u)\ d^3u d^2 \omega,$$ and $K_{1}(f)$ is given by
\begin{equation}\label{eq:difK1}
\begin{array}{lll}
K_{1}(f)&:=& 2\pi \int_{\mathbb{R}^3} |u-v|^{-1} e^{ -\frac{|(u-v)\cdot v|^2}{|u-v|^2}} f(u)\ d^3 u-\pi \int_{\mathbb{R}^3} |u-v|e^{-|v|^2} f(u)\ d^3 u\\
& &\\
&=&\int_{\mathbb{R}^{3}\times \mathbb{S}^2} |(u-v)\cdot\omega| M(u) f(u)\ d^3 ud^2 \omega\\
& &\\
& &- \int_{\mathbb{R}^3\times\mathbb{S}^2 }|(u-v)\cdot\omega |M(u') f(v')\ d^3 ud^2 \omega\\
& &\\
& &-\int_{\mathbb{R}^3 \times \mathbb{S}^2} |(u-v)\cdot\omega| M(v') f(u')\ d^3 ud^2 \omega
\end{array}
\end{equation}
The explicit form of $K_{1}$ has been derived by R.T.Glassey in ~\cite{Glassey1996}, (see also ~\cite{MR0156656, CourHil1989}).
\end{itemize}

To simplify our notations, we define operators $K$ and $\nu$ by
\begin{equation}\label{eq:difK}
K:=-K_{0}+C_{\infty}\kappa K_{1}
\end{equation}
\begin{equation}\label{eq:difNu}
\nu:=\nu_0+C_{\infty}\kappa \nu_1.
\end{equation} Then the linear operator $L$ in ~\eqref{eq:ALinear} is given by $$L=\nu+v\cdot\nabla_{x}+K.$$

To prepare the ground for our analysis, we state some estimates on the nonlinearity $Q$ and the operators $\nu,$ $K_{0}$ and $K_1$. These estimates show that all these operators are unbounded.
\begin{lemma}\label{LM:EstNonline}
There exists a positive constant $\Lambda$ such that
\begin{equation}\label{eq:globalLower}
\nu_0(v),\ \nu_1(v)\geq \Lambda(1+|v|).
\end{equation}
For any $m\geq 0,$ there exists a constant $C_{m}$ such that, for arbitrary functions $f,\  g
\in L^{1}(\mathbb{R}^{3}),$
\begin{equation}\label{eq:estK0}
\|\langle v\rangle^{m}K_{0}f\|_{L^{1}(\mathbb{R}^{3})}\leq C_{m} \|\langle v\rangle f\|_{L^{1}(\mathbb{R}^{3})},
\end{equation}
\begin{equation}\label{eq:mK1}
\|\langle v\rangle^{m}K_{1}f\|_{L^{1}(\mathbb{R}^{3})}\leq C_{m} \|\langle v\rangle^{m+1}f\|_{L^{1}(\mathbb{R}^{3})},
\end{equation}
and
\begin{equation}\label{eq:estNonL}
\begin{array}{lll}
\|\langle v\rangle^{m} Q(f,g)\|_{L^{1}(\mathbb{R}^{3})}
&\leq& C_{m} \|f\|_{L^{1}(\mathbb{R}^{3})}\|\langle v\rangle^{m+1}g\|_{L^{1}(\mathbb{R}^{3})}+ C_{m}\|\langle v\rangle^{m+1}f\|_{L^{1}(\mathbb{R}^{3})}\|g\|_{L^{1}(\mathbb{R}^{3})}.
\end{array}
\end{equation}
\end{lemma}
This lemma is proven in Appendix ~\ref{sec:nonlinear}.
\section{Local Well-Posedness of Equation ~\eqref{eq:NLBL1}}\label{sec:LocalWell}
In this section we prove local wellposedness of equation ~\eqref{eq:NLBL1}.

We briefly present the ideas used in the proof. One of the difficulties tackled in the present paper is that the nonlinearity $Q(f,f)$ is unbounded; see ~\eqref{eq:estNonL}. To overcome it we adopt a technique drawn from the works ~\cite{YGuo2002,YGuo2003}. Specifically, we consider the solution $f$ in a Banach space to be defined in ~\eqref{eq:Norm} below, the second term in its definition playing a crucial role in controlling $Q(f,f)$. For computational details we refer to ~\eqref{eq:integrationP} below.

The main result of this section is
\begin{proposition}\label{Prop:wellposed}
If the constant $\kappa>0$ in ~\eqref{eq:NLBL1} is sufficiently small and if $$\sum_{|\alpha|\leq 8} \|\langle v\rangle^{m} \partial_{x}^{\alpha} f_0\|_{L^{1}(\mathbb{R}^{3}\times \mathbb{T}^{3})}\leq \kappa^{-\frac{1}{4}},$$
for some $m\geq 2,$ then there exists a constant $T=T(\kappa)$ such that, on the interval $[0,T],$ equation ~\eqref{eq:ALinear} has a unique solution $f$ satisfying $$\sum_{|\alpha|\leq 8} \|\langle v\rangle^{m} \partial_{x}^{\alpha} f(\cdot,t)\|_{L^{1}(\mathbb{R}^{3}\times \mathbb{T}^{3})}+ \int_{0}^{t} \|\langle v\rangle^{m+1} \partial_{x}^{\alpha}f(\cdot,s)\|_{L^{1}(\mathbb{R}^{3}\times \mathbb{T}^{3})}\ ds\leq \kappa^{-\frac{1}{2}}.$$
\end{proposition}
\begin{proof}
To simplify the notation we denote $L^1(\mathbb{R}^{3}\times \mathbb{T}^{3})$ by $L^1$.

To recast ~\eqref{eq:ALinear} in a convenient form,
we rewrite this equation using Duhamel's principle,
\begin{equation}
f(t)=e^{-t[\nu+v\cdot\nabla_{x}]}f_0+ \int_{0}^{t} e^{-(t-s)(\nu+v\cdot\nabla_{x})} H(f(s))\ ds,
\end{equation} with $H(f):=-K_{0}f+C_{\infty}\kappa K_{1}f+\kappa Q(f,f)$.

In order to be able to apply suitable results of functional analysis, we demand that $f$ and the terms on the right hand side belong to a suitable Banach space. We define a family of Banach spaces, $\mathcal{B}_{\delta}$, $0<\delta\ll 1,$ by
$$
\mathcal{B}_{\delta}:=\{g:\mathbb{R}^{3}\times \mathbb{T}^{3}\times [0,\delta] \rightarrow \mathbb{C}| \ \|g\|_{\mathcal{B}_{\delta}}<\infty\}
$$ where $\|g\|_{\mathcal{B}_{\delta}}$ is defined by
\begin{equation}\label{eq:Norm}
\|g\|_{\mathcal{B}_{\delta}}:=\sum_{|\alpha|\leq 8}[\sup_{0\leq s\leq \delta}\|\langle v\rangle^{m}\partial_{x}^{\alpha} g(\cdot,s)\|_{L^{1}}+\int_{0}^{\delta} \|\langle v\rangle^{m+1}\partial_{x}^{\alpha}g(\cdot, s)\|_{L^{1}}\ ds].
\end{equation}

Our key observations are:
\begin{itemize}
\item[(1)]
\begin{equation}\label{eq:nmIni}
\|e^{-t[\nu+v\cdot \nabla_{x}]} f_{0}\|_{\mathcal{B}_{\delta}}\leq \displaystyle\sum_{|\alpha|\leq 8}\|\langle v\rangle^{m}\partial_{x}^{\alpha}f_{0}\|_{L^{1}};
\end{equation}
\item[(2)] We define a nonlinear map, $\Pi,$ by $$\Pi(f,t):=  \int_{0}^{t} e^{-(t-s)(\nu+v\cdot\nabla_{x})} H(f(s))\ ds.$$ Then $\Pi: \ \mathcal{B}_{\delta}\rightarrow \mathcal{B}_{\delta}$ is a contractive map if restricted to a suitable domain. More specifically,
\begin{equation}\label{eq:contractive}
\|\Pi(f)-\Pi(g)\|_{\mathcal{B}_{\delta}}\leq \frac{1}{2} \|f-g\|_{\mathcal{B}_{\delta}},
\end{equation}
provided $ \|f\|_{\mathcal{B}_{\delta}},\ \|g\|_{\mathcal{B}_{\delta}}\leq \kappa^{-\frac{1}{4}},$ and for $\delta$ sufficiently small.
\end{itemize}

Obviously these two results, ~\eqref{eq:nmIni} and ~\eqref{eq:contractive}, together with the contraction lemma, imply the existence of a unique solution in the time interval $[0,\delta],$ provided that $\delta=\delta(\kappa)$ is sufficiently small.

In what follows we prove ~\eqref{eq:nmIni} and ~\eqref{eq:contractive}.

To prove ~\eqref{eq:contractive}, we start by estimating $\|\langle v\rangle^{m}\partial_{x}^{\alpha}[\Pi(f)-\Pi(g)]\|_{L^{1}}, \ |\alpha|\leq 8.$ We decompose this quantity into three terms:
\begin{equation}\label{eq:FirstStep}
\begin{array}{lll}
& &\|\langle v\rangle^{m}\int_{0}^{t} e^{-(t-s)(\nu+v\cdot \nabla_{x})} \partial_{x}^{\alpha}[H(f(s))-H(g(s))]\ ds\|_{L^{1}}\\
& &\\
&\leq &\int_{0}^{t} \|\langle v\rangle^{m}e^{-(t-s)(\nu+v\cdot \nabla_{x})} \partial_{x}^{\alpha}[H(f(s))-H(g(s))]\|_{L^{1}} ds\\
& &\\
&\leq &\int_{0}^{t} \|\langle v\rangle^{m} \partial_{x}^{\alpha}[H(f(s))-H(g(s))]\|_{L^{1}} ds\\
& &\\
&\lesssim & \int_{0}^{t} \|\langle v\rangle^{m}\partial_{x}^{\alpha}K_{0}(f(s)-g(s))\|_{L^{1}}\ ds+\kappa C_{\infty} \int_{0}^{t} \|\langle v\rangle^{m}K_{1}\partial_{x}^{\alpha}(f(s)-g(s))\|_{L^{1}}\ ds\\
& &\\
& &+\kappa\int_{0}^{t}\|\langle v\rangle^{m}\partial_{x}^{\alpha}[Q(f,f)(s)-Q(g,g)(s)]\ \|_{L^{1}}\ ds\\
& &\\
&=&\Psi_1+\Psi_2+\Psi_{3},
\end{array}
\end{equation}
where in the third step we use the fact that the operator $e^{t v\cdot \nabla_{x}}$ preserves $L^1$ norm. The terms $\Psi_{k},\ k=1,2,3,$ are defined in the obvious manner and are estimated below.
\begin{enumerate}
\item[(1)]
By Lemma ~\ref{LM:EstNonline}, inequality ~\eqref{eq:estK0},
\begin{equation}\label{eq:estPsi1}
\begin{array}{lll}
\Psi_{1}&=&\int_{0}^{t} \|\langle v\rangle^{m}K_{0}\partial_{x}^{\alpha}(f(s)-g(s))\|_{L^{1}}\ ds\\
& &\\
&\lesssim & \int_{0}^{t} \|\langle v\rangle \partial_{x}^{\alpha}(f-g)\|_{L^{1}}\ ds\\
& &\\
&\leq& t\|f-g\|_{\mathcal{B}_{\delta}},\ \text{for}\ t\leq \delta.
\end{array}
\end{equation}
\item[(2)] From Lemma ~\ref{LM:EstNonline}, inequality ~\eqref{eq:mK1} we deduce that
\begin{equation}\label{eq:estPsi2}
\Psi_{2}\lesssim \kappa C_{\infty} \int_{0}^{t} \|\langle v\rangle^{m+1}\partial_{x}^{\alpha}(f(s)-g(s))\|_{L^{1}}\ ds\lesssim \kappa C_{\infty} \|f-g\|_{\mathcal{B}_{\delta}},
\end{equation} for $t\leq \delta$.
\item[(3)] To estimate $\Psi_{3},$ we use the definition of $Q$ to obtain
$$\Psi_{3}\leq \kappa \int_{0}^{t} \|\langle v\rangle^{m}\partial_{x}^{\alpha}Q(f-g,f)(s)\|_{L^{1}}+ \|\langle v\rangle^{m}\partial_{x}^{\alpha}Q(g,f-g)(s)\|_{L^{1}}\  ds.$$
Using ~\eqref{eq:AppliEmbed}, below, we find that
\begin{equation}\label{eq:estPsi3}
\begin{array}{lll}
\Psi_{3}&\leq &\displaystyle\kappa \int_{0}^{t} \{\sum_{|\beta_1|\leq 8}\|\langle v\rangle^{m+1}\partial_{x}^{\beta_1}(f-g)\|_{L^{1}}\sum_{|\beta_2|\leq 8}(\|\partial_{x}^{\beta_2} f\|_{L^{1}}+\|\partial_{x}^{\beta_2}g\|_{L^{1}})\ \} ds\\
& &\\
& &+\kappa \displaystyle\int_{0}^{t}\{ \sum_{|\beta_1|\leq 8}\|\partial_{x}^{\beta_1}(f-g)\|_{L^{1}}\sum_{|\beta_2|\leq 8}(\|\langle v\rangle^{m+1} \partial_{x}^{\beta_2}f\|_{L^{1}}+\|\langle v\rangle^{m+1}\partial_{x}^{\beta_2}g\|_{L^{1}})\ \} ds\\
& &\\
&\lesssim&\kappa \|f-g\|_{\mathcal{B}_{\delta}}[\|f\|_{\mathcal{B}_{\delta}}+\|g\|_{\mathcal{B}_{\delta}}].
\end{array}
\end{equation}
\end{enumerate}
Collecting these estimates, we conclude, that for any $t\leq \delta\ll 1,$
\begin{equation}\label{eq:halfBdelta}
\|\langle v\rangle^{m}\partial_{x}^{\alpha}[\Pi(f)-\Pi(g)]\|_{L^{1}}\lesssim \|f-g\|_{\mathcal{B}_{\delta}}[C_{\infty}\kappa +\kappa(\|f\|_{\mathcal{B}_{\delta}}+\|g\|_{\mathcal{B}_{\delta}}) +\delta].
\end{equation}

Next, we estimate $\int_{0}^{t}\|\langle v\rangle^{m+1}\partial_{x}^{\alpha}[\Pi(f(s))-\Pi(g(s))]\|_{L^{1}}\ ds$. By direct computation,
\begin{equation}\label{eq:integrationP}
\begin{array}{lll}
& &\int_{0}^{t} \|\int_{0}^{s} e^{-(s-s_1)(\nu+v\cdot \nabla_{x})} \langle v\rangle^{m+1}\partial_{x}^{\alpha}[H(f(s_1))-H(g(s_1))] \ ds_1 \|_{L^{1}}\ ds\\
& &\\
&\leq &\int_{0}^{t}\int_{0}^{s} \|e^{-(s-s_1)(\nu+v\cdot\nabla_{x})}\langle v\rangle^{m+1}\partial_{x}^{\alpha} [H(f(s_1))-H(g(s_1))] \|_{L^{1}} \ ds_1 ds\\
& &\\
&= &\int_{0}^{t}\int_{0}^{s} \|e^{-(s-s_1)\nu}\langle v\rangle^{m+1}\partial_{x}^{\alpha} [H(f(s_1))-H(g(s_1))] \|_{L^{1}} \ ds_1 ds\\
& &\\
&=&\|\int_{0}^{t}\int_{0}^{s}  e^{-(s-s_1)\nu} |\langle v\rangle^{m+1}\partial_{x}^{\alpha}[H(f(s_1))-H(g(s_1)) ]| \ ds_1 ds \|_{L^{1}}\\
& &\\
&\leq & \|\int_{0}^{t} \nu^{-1}\langle v\rangle^{m+1} |\partial_{x}^{\alpha}[H(f(s))-H(g(s))] | \ ds \|_{L^{1}}\\
& &\\
&\lesssim &\int_{0}^{t} \|\langle v\rangle^{m} \partial_{x}^{\alpha}[H(f(s))-H(g(s))]  \|_{L^{1}} ds,
\end{array}
\end{equation} where the crucial step is the fourth one and is accomplished by integrating by parts in the variable $s,$ the last inequality results from our estimate on $\nu=\nu_0+C_{\infty}\kappa \nu_1$ in ~\eqref{eq:globalLower}. Here the condition on $r_0$ being unbounded, (see ~\eqref{eq:difNu0}), is used.

We observe that the last step in ~\eqref{eq:integrationP} is the same to that in the third line of ~\eqref{eq:FirstStep}. Hence it also admits the estimate in ~\eqref{eq:halfBdelta}, i.e.,
\begin{equation}\label{eq:halfBdelta2}
\int_{0}^{t}\|\langle v\rangle^{m+1}\partial_{x}^{\alpha}[\Pi(f(s))-\Pi(g(s))]\|_{L^{1}}\ ds\lesssim \|f-g\|_{\mathcal{B}_{\delta}}[C_{\infty}\kappa+\delta +\kappa(\|f\|_{\mathcal{B}_{\delta}}+\|g\|_{\mathcal{B}_{\delta}}) ].
\end{equation}
This, together with ~\eqref{eq:halfBdelta}, implies ~\eqref{eq:contractive}.

Next, we prove ~\eqref{eq:nmIni}. By direct computation $$
\begin{array}{lll}
\|e^{-t[\nu+v\cdot \nabla_{x}]} f_{0}\|_{\mathcal{B}_{\delta}}&=&\displaystyle\sum_{|\alpha|\leq 8}[\displaystyle\sup_{0\leq t\leq \delta}\|\langle v\rangle^{m}e^{-t[\nu+v\cdot\nabla_{x}]}\partial_{x}^{\alpha}f_0\|_{L^{1}}+\int_{0}^{t}\|\langle v\rangle^{m+1} e^{-s(\nu+v\cdot\nabla_{x})} \partial_{x}^{\alpha}f_0\|_{L^{1}}\ ds]\\
& &\\
&=&\displaystyle\sum_{|\alpha|\leq 8}[\|\langle v\rangle^{m}\partial_{x}^{\alpha}f_0\|_{L^{1}}+\|\int_{0}^{t} e^{-s\nu } \ ds\ \langle v\rangle^{m+1} |\partial_{x}^{\alpha}f_0|\|_{L^{1}}]\\
& &\\
&\leq & \displaystyle\sum_{|\alpha|\leq 8}\|\langle v\rangle^{m}\partial_{x}^{\alpha}f_0\|_{L^{1}}
\end{array}
$$ which is ~\eqref{eq:nmIni}.
\end{proof}

In the proof we have used the following embedding results; (see ~\eqref{eq:estPsi3}).
\begin{lemma}\label{LM:embedding} For any function $f:\ \mathbb{R}^{3}\times \mathbb{T}^{3}\rightarrow \mathbb{C},$
\begin{equation}\label{eq:Poincare}
\sup_{x\in \mathbb{Z}^{3}} \|f(\cdot, x)\|_{L^{1}(\mathbb{R}^{3})}\leq C\sum_{|\alpha|\leq 4} \|\partial_{x}^{\alpha}f\|_{L^{1}(\mathbb{R}^{3}\times \mathbb{T}^{3})}.
\end{equation} For any $\alpha\in (\mathbb{Z}^{+})^{3}$ satisfying $|\alpha|\leq 8$, and for arbitrary functions $f,\ g:\ \mathbb{R}^{3}\times \mathbb{T}^{3}\rightarrow \mathbb{C},$
\begin{equation}\label{eq:AppliEmbed}
\begin{array}{lll}
\|\langle v\rangle^{m}\partial_{x}^{\alpha} Q(f,g)\|_{L^{1}(\mathbb{R}^{3}\times \mathbb{T}^{3})}&\lesssim &\displaystyle\sum_{|\beta_1|, \ |\beta_2|\leq 8} [\|\langle v\rangle^{m+1}\partial_{x}^{\beta_1}f\|_{L^{1}(\mathbb{R}^{3}\times \mathbb{T}^{3})} \|\partial_{x}^{\beta_2}g\|_{L^{1}(\mathbb{R}^{3}\times \mathbb{T}^{3})}\\
& &\\
&+&\|\langle v\rangle^{m+1}\partial_{x}^{\beta_1}g\|_{L^{1}(\mathbb{R}^{3}\times \mathbb{T}^{3})} \|\partial_{x}^{\beta_2}f\|_{L^{1}(\mathbb{R}^{3}\times \mathbb{T}^{3})}].
\end{array}
\end{equation}
\end{lemma}
\begin{proof}
We start with the proof of ~\eqref{eq:Poincare}. We Fourier-expand the function $f\in L^{1}(\mathbb{R}^{3}\times \mathbb{T}^{3})$ in the variable $x$:
$$f(v,x)= \sum_{{\bf{n}}\in \mathbb{Z}^{3}}e^{i{\bf{n}}\cdot x} f_{{\bf{n}}}(v)$$ with $f_{{\bf{n}}}(v):=\frac{1}{(2\pi)^{3}} \int_{\mathbb{T}^{3}} f(v,x) e^{-i{\bf{n}}\cdot x}\ dx.$ Obviously
\begin{equation}\label{eq:rewrite}
\|f(\cdot,x)\|_{L^{1}(\mathbb{R}^{3})} \lesssim \sum_{{\bf{n}}\in \mathbb{Z}^{3}} \|f_{{\bf{n}}}\|_{L^{1}(\mathbb{R}^{3})} .
\end{equation}

We now write $\|f_{{\bf{n}}}\|_{L^{1}(\mathbb{R}^3)}$ as a product of $\frac{1}{(1+|{\bf{n}}|)^4}$ and $(1+|{\bf{n}}|)^{4}\|f_{{\bf{n}}}\|_{L^{1}(\mathbb{R}^3)}.$
To control the factor $(1+|{\bf{n}}|)^{4}f_{{\bf{n}}}$ in $L^{1}$ we use the observation that
$$
\begin{array}{lll}
(1+|{\bf{n}}|^{4}) |f_{n}| &=&(1+|{\bf{n}}|^{4}) (2\pi)^{-3} |\langle f,\ e^{i{\bf{n}}\cdot x}\rangle_{\mathbb{T}^{3}}|\\
& &\\
&\lesssim & \displaystyle\sum_{|\alpha|\leq 4} |\langle \partial_{x}^{\alpha} f,\ e^{i{\bf{n}}\cdot x}\rangle_{\mathbb{T}^{3}}|
\end{array}
$$
to obtain that
\begin{equation}\label{eq:embd}
(1+|{\bf{n}}|)^{4} \|f_{n}\|_{L^{1}}\lesssim  \sum_{|\alpha|\leq 4} \|\partial_{x}^{\alpha}f\|_{L^{1}(\mathbb{R}^{3}\times \mathbb{T}^{3})}.
\end{equation}
This, together with the fact that
$
\displaystyle\sum_{{\bf{n}}\in \mathbb{Z}^{3}} \frac{1}{(1+|{\bf{n}}|)^{4}} <\infty
$ and with ~\eqref{eq:rewrite}, implies the desired estimate.

Next we prove ~\eqref{eq:AppliEmbed}. It is easy to see that $$\langle v\rangle^{m}|\partial_{x}^{\alpha}Q(f,g)|\lesssim \sum_{\beta_1+\beta_2=\alpha}|\langle v\rangle^{m}Q(\partial_{x}^{\beta_1}f, \ \partial_{x}^{\beta_2}g)|.$$
Obviously
\begin{equation}
\|\langle v\rangle^{m}Q(\partial_{x}^{\beta_1}f, \ \partial_{x}^{\beta_2}g)\|_{L^{1}(\mathbb{R}^{3}\times \mathbb{T}^{3})}=\|\ \|\langle v\rangle^{m}Q(\partial_{x}^{\beta_1}f, \ \partial_{x}^{\beta_2}g)\|_{L^{1}(\mathbb{R}^{3})}\|_{L^{1}(\mathbb{T}^{3})}.
\end{equation} We apply ~\eqref{eq:estNonL} to obtain
\begin{equation}\label{eq:nonLest}
\|\langle v\rangle^{m}Q(\partial_{x}^{\beta_1}f, \ \partial_{x}^{\beta_2}g)\|_{L^{1}(\mathbb{R}^{3})}\lesssim \|\langle v\rangle^{m+1}\partial_{x}^{\beta_1}f\|_{L^{1}(\mathbb{R}^{3})}\|\partial_{x}^{\beta_2}g\|_{L^{1}(\mathbb{R}^{3})}+\|\langle v\rangle^{m+1}\partial_{x}^{\beta_1}g\|_{L^{1}(\mathbb{R}^{3})}\|\partial_{x}^{\beta_2}f\|_{L^{1}(\mathbb{R}^{3})}.
\end{equation}
In the next we estimate the right hand side of ~\eqref{eq:nonLest} in $L^{1}(\mathbb{T}^3).$ Since $\beta_1+\beta_2=\alpha$ and $|\alpha|\leq 8,$ at least one of $|\beta_1|,\ |\beta_{2}|$ is less than or equal to $4$. Without loss of generality, we assume that $|\beta_1|\leq 4.$ Applying ~\eqref{eq:Poincare} to the first term on the right hand side we find that
\begin{equation}
\begin{array}{lll}
\|\ \|\langle v\rangle^{m+1}\partial_{x}^{\beta_1}f\|_{L^{1}(\mathbb{R}^{3})}\|\partial_{x}^{\beta_2}g\|_{L^{1}(\mathbb{R}^{3})}
\|_{L^{1}(\mathbb{T}^{3})}&\leq&\displaystyle\max_{x\in \mathbb{T}^3}\|\langle v\rangle^{m+1}\partial_{x}^{\beta_1}f\|_{L^{1}(\mathbb{R}^{3})} \|\partial_{x}^{\beta_2}g\|_{L^{1}(\mathbb{R}^{3}\times \mathbb{T}^3)}\\
& &\\
&\lesssim &\displaystyle\sum_{|\beta_1|, \ |\beta_2|\leq 8} \|\langle v\rangle^{m+1}\partial_{x}^{\beta_1}f\|_{L^{1}(\mathbb{R}^{3}\times \mathbb{T}^{3})} \|\partial_{x}^{\beta_2}g\|_{L^{1}(\mathbb{R}^{3}\times \mathbb{T}^{3})}
\end{array}
\end{equation}
The second term on the right hand side can be estimated almost identically.

Collecting the estimates above we complete the proof of ~\eqref{eq:AppliEmbed}.

\end{proof}
\section{Propagator Estimates}\label{sec:propagatorEst}
Recall the definition of the linear operator $L$ in ~\eqref{eq:difL}. In this section, we study decay estimates of the operator $e^{-tL}(1-P_{0})$ acting on $L^{1}$, where $P_{0}: L^{1}(\mathbb{R}^{3}\times \mathbb{T}^{3})\rightarrow L^{1}(\mathbb{R}^{3}\times \mathbb{T}^{3})$ is the Riesz projection onto the 0-eigenspace, $\{e^{-|v|^2}\}:$
\begin{equation}\label{eq:difProjection}
P_{0}f:= \frac{1}{8\pi^{\frac{7}{2}}} e^{-|v|^2} \int_{\mathbb{R}^{3}\times \mathbb{T}^{3}} f(v,x)\ d^3 v d^3 x.
\end{equation}

The main theorem of this section is
\begin{theorem}\label{THM:propagator}
There exist constants $C_0,\ C_{1}>0$ and an integer $m<\infty$ such that
\begin{equation}
\|e^{-tL}(1-P_0)g\|_{L^{1}(\mathbb{R}^{3}\times \mathbb{T}^{3})}\leq C_{1}e^{-C_0t} \|\langle v\rangle^{m}g\|_{L^{1}(\mathbb{R}^{3}\times \mathbb{T}^{3})}.
\end{equation}
\end{theorem}

We first outline the general strategy of the proof.

There are two typical approaches to proving decay estimates for propagators. The first one is to apply the spectral theorem, (see e.g. ~\cite{RSI}), to obtain $$e^{-tL}(1-P_0)=\frac{1}{2\pi i}\oint_{\Gamma} e^{-t\lambda}(\lambda -L)^{-1}\ d\lambda$$ where the contour $\Gamma$ is a curve encircling the spectrum of $L(1-P_0).$ The obstacle is that the spectrum of $L(1-P_0)$ occupies the entire right half of the complex plane, except for a strip in a neighborhood of the imaginary axis, as illustrated in Figure ~\ref{fig:FigureExample} below. This makes it difficult to prove strong convergence on $L^1$ of the integral on the right hand side.
\begin{figure}[htb!]
\centering%
\includegraphics[width=13cm]{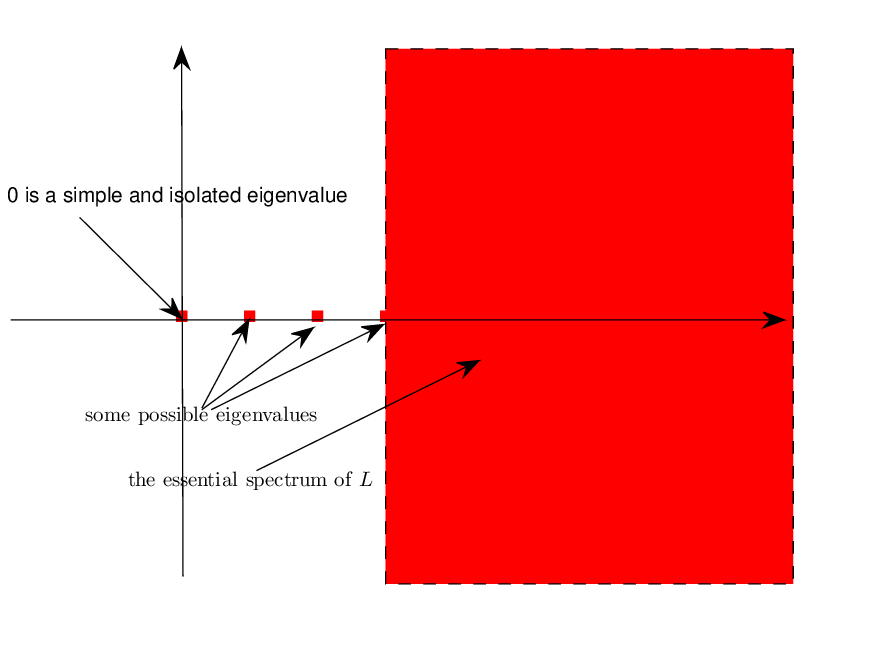}
\caption{The Spectrum of $L$}
\label{fig:FigureExample}
\end{figure}

The second approach is to use perturbation theory, which amounts to expanding $e^{-tL}$ in powers of the operator $K,$ (see ~\eqref{eq:difK}): $$e^{-tL}=e^{-t(\nu+v\cdot\nabla_{x})}+\int_{0}^{t}e^{-(t-s)(\nu+v\cdot\nabla_x)}K e^{-s(\nu+v\cdot\nabla_x)}\ ds+\cdots.$$ It will be shown in Proposition ~\ref{prop:easyEst} that each term in this expansion can be estimated quite well, but the fact that $K$ is unbounded forces us to estimate them in different spaces.

We will combine these two approaches to prove Theorem ~\ref{THM:propagator}.

We expand the propagator $e^{-tL}(1-P_0)$ using Duhamel's principle:
\begin{equation}\label{eq:durha}
e^{-tL}(1-P_0)= \sum_{k=0}^{12} (1-P_0)A_{k}(t)+(1-P_0)\tilde{A}(t),
\end{equation}
where the operators $A_{k}$ are defined recursively, with
\begin{equation}\label{eq:difA0}
A_{0}=A_{0}(t):= e^{-t (\nu+v\cdot \nabla_{x})},
\end{equation} and $A_{k},\ k=1,2,\cdots,12,$ given by
\begin{equation}
A_{k}(t):= \int_{0}^{t} e^{-(t-s)(\nu+v\cdot \nabla_{x})}K A_{k-1}(s)\ ds .
\end{equation}
Finally $\tilde{A}$ is defined by
\begin{equation}
\tilde{A}(t)=\int_{0}^{t} e^{-(t-s)L}K A_{12}(s)\ ds.
\end{equation}
The exact form of $A_{k},\ k=0,\ 1,\cdots, 12,$ implies the following estimates.
\begin{proposition}\label{prop:easyEst}
There exist positive constants $C_0$ and $C_1$ such that, for any function $f: \ \mathbb{R}^{3}\times \mathbb{T}^{3}\rightarrow \mathbb{C},$
\begin{equation}\label{eq:explEst}
\|A_{k}(t)f\|_{L^{1}(\mathbb{R}^{3}\times \mathbb{T}^3)}\leq C_1 e^{-C_0 t} \|\langle v\rangle^{k}f\|_{L^{1}(\mathbb{R}^{3}\times \mathbb{T}^3)}.
\end{equation}
\end{proposition}
This proposition is proven in Subsection ~\ref{subsec:ProofEEst}.

The estimate on $\tilde{A}$, which, by definition, is given by
$$\tilde{A}=\int_{0}^{t} e^{-(t-s_1)L}K \int_{0}^{s_1} e^{-(s_1-s_2) (\nu+v\cdot\nabla_{x})}K\cdots \int_{0}^{s_{12}}e^{-(s_{12}-s_{13})(\nu+v\cdot\nabla_{x})}K e^{-s_{13}(\nu+v\cdot\nabla_{x})}\ ds_{13}\cdots ds_1, $$ is more involved.

We first transform $\tilde{A}$ to a more convenient form.

One of the important properties of the operators $L$ and $L_{0}$ is that, for any function $g:\ \mathbb{R}^{3}\rightarrow \mathbb{C}$ (i.e., independent of $x$) and ${\bf{n}}\in \mathbb{Z}^3,$ we have that
\begin{equation}\label{eq:observation}
\begin{array}{lll}
P_{0}e^{i{\bf{n}}\cdot x}g&=&0\ \ \text{if}\ {\bf{n}}\not=0,\\
& &\\
L e^{i{\bf{n}}\cdot x}g&=&e^{i{\bf{n}}\cdot x}L_{{\bf{n}}}g,\\
& &\\
(\nu+v\cdot\nabla_x)e^{i{\bf{n}}\cdot x}g&=&e^{i{\bf{n}}\cdot x}(\nu+i{\bf{n}}\cdot v)g,
\end{array}
\end{equation} where the operator $L_{{\bf{n}}}$ is unbounded and defined as $$L_{{\bf{n}}}:=\nu+i{\bf{n}}\cdot v+K.$$ (Recall that $P_0$ has been defined in ~\eqref{eq:difProjection}.)

To make ~\eqref{eq:observation} applicable, we Fourier-expand the function $g:\ \mathbb{R}^{3}\times \mathbb{T}^{3}\rightarrow \mathbb{C}$ in the variable $x,$ i.e.,
\begin{equation}\label{eq:dirSum}
g(v,x)=\sum_{{\bf{n}}\in \mathbb{Z}^{3}} e^{i{\bf{n}}\cdot x} g_{{\bf{n}}}(v).
\end{equation} Then ~\eqref{eq:rewrite} and ~\eqref{eq:observation} yield the bound
\begin{equation}\label{eq:decomA}
\|(1-P_0)\tilde{A}g\|_{L^{1}(\mathbb{R}^3\times \mathbb{T}^3)}\leq \sum_{{\bf{n}}\in \mathbb{Z}^{3}}\|\tilde{A}_{{\bf{n}}} g_{n}\|_{L^{1}(\mathbb{R}^3)},
\end{equation} where $\tilde{A}_{{\bf{n}}}$ is defined as follows: If ${\bf{n}}\not=(0,0,0)$ then
$$
\tilde{A}_{{\bf{n}}}:=\int_{0}^{t} e^{-(t-s_1)L_{{\bf{n}}}}K \int_{0}^{s_1} e^{-(s_1-s_2) (\nu+i v\cdot{\bf{n}})}K\cdots \int_{0}^{s_{12}}e^{-(s_{12}-s_{13})(\nu+iv\cdot{\bf{n}})}K e^{-s_{13}(\nu+iv\cdot{\bf{n}})}\ ds_{13}\cdots ds_1
$$ and for ${\bf{n}}=(0,0,0)$ we define $$
\tilde{A}_{0}:=\int_{0}^{t} (1-P_0)e^{-(t-s_1)L_{0}}K \int_{0}^{s_1} e^{-(s_1-s_2) \nu}K\cdots \int_{0}^{s_{12}}e^{-(s_{12}-s_{13})\nu}K e^{-s_{13}\nu}\ ds_{13}\cdots ds_1.
$$

Next, we study $\tilde{A}_{{\bf{n}}}$, which is defined in terms of the operators $e^{-tL_{{\bf{n}}}}$, $e^{-t[\nu+i {\bf{n}}\cdot v]}$ and $Ke^{-t[\nu+i{\bf{n}}\cdot v]}K.$

It is easy to estimate $e^{-t[\nu+i{\bf{n}}\cdot v]}:$ The fact that the function $\nu$ has a positive global minimum $\Lambda$ (see ~\eqref{eq:globalLower}) implies that
\begin{equation}\label{eq:exactForm}
 \|e^{-t[\nu+i{\bf{n}}\cdot v]}\|_{L^{1}\rightarrow L^1}\leq e^{-\Lambda t}.
\end{equation}
We provide some rough estimate on the operator $e^{-tL_{{\bf{n}}}}$.
\begin{lemma}\label{LM:roughEst}
If ${\bf{n}}\not= (0,0,0)$ then there exist positive constants $C_0$ and $C_1$ such that
\begin{equation}\label{eq:est000}
\|e^{-tL_{{\bf{n}}}}\|_{L^{1}(\mathbb{R}^{3})\rightarrow L^{1}(\mathbb{R}^{3})} \leq C_{1}(1+|{\bf{n}}|) e^{-C_{0}t}.
\end{equation} For ${\bf{n}}=(0,0,0)$
\begin{equation}\label{eq:estNot0}
\|e^{-tL_{{\bf{0}}}}(1-P_{0})\|_{L^{1}(\mathbb{R}^{3})\rightarrow L^{1}(\mathbb{R}^{3})} \leq C_{1} e^{-C_{0}t}.
\end{equation}
\end{lemma}
This lemma will be proven in Subsection ~\ref{subsec:RoughEst}.

The most important step is to estimate $$K_{t}^{({\bf{n}})}:= Ke^{-t(\nu+i{\bf{n}}\cdot v)} K.$$ Let $K(v,u)$ be the integral kernel of $K$. Then the integral kernel, $K_{t}^{({\bf{n}})}(v,u),$ of $K_{t}^{({\bf{n}})}$ is given by
$$K_{t}^{({\bf{n}})}(v,u)=\int_{\mathbb{R}^{3}}K(v,z)e^{-t[\nu(z)+i{\bf{n}}\cdot z]}K(z,u)\ dz.$$ The presence of the factor $e^{-it{\bf{n}}\cdot z}$ plays an important role. It makes the operator $K_{t}^{({\bf{n}})}$ smaller, as $|{\bf{n}}|$ becomes larger.
\begin{lemma}\label{LM:oscillate}
There exist positive constants $C_0$ and $C_{1}$ such that, for any ${\bf{n}}\in \mathbb{Z}^{3},$
\begin{equation}\label{eq:oscilate}
\|K_{t}^{({\bf{n}})} f\|_{ L^{1}(\mathbb{R}^{3})} \leq \frac{C_1}{1+|{\bf{n}}|t} e^{-C_0 t}\|\langle v\rangle^{3} f\|_{L^{1}(\mathbb{R}^{3})}.
\end{equation}
\end{lemma}
This lemma will be proven in Subsection ~\ref{subsec:oscillate}.

The results in Proposition ~\ref{prop:easyEst}, Lemma ~\ref{LM:roughEst} and Lemma ~\ref{LM:oscillate} suffice to prove Theorem ~\ref{THM:propagator}.\\
{\bf{Proof of Theorem ~\ref{THM:propagator}}.} In Equation ~\eqref{eq:durha} we have decomposed $e^{-tL}(1-P_0)$ into several terms. The operators $A_{k}, \ k=0,1,2,\cdots, 12,$ are estimated in Proposition ~\ref{prop:easyEst}.

In what follows, we study $\tilde{A}$. By ~\eqref{eq:decomA} we only need to control $\tilde{A}_{{\bf{n}}},\ {\bf{n}}\in \mathbb{Z}^3$. For ${\bf{n}}=(0,0,0)$ it is easy to see that
\begin{equation}\label{eq:TildeA0}
\|\tilde{A}_{0}g_{{\bf{0}}}\|_{L^{1}(\mathbb{R}^{3})}\lesssim e^{-C_0 t} \|\langle v\rangle^{12} g_{\bf{0}}\|_{L^{1}(\mathbb{R}^{3})}
\end{equation} by collecting the different estimates in ~\eqref{eq:exactForm} and Lemma ~\ref{LM:roughEst} and using the estimates on $K=-K_0+C_{\infty}\kappa K_1$ in Lemma ~\ref{LM:EstNonline}.

For ${\bf{n}}\not=0$, we observe that the integrands in the definitions of $\tilde{A}_{{\bf{n}}}$ are products of terms $e^{-(t-s_1)L_{{\bf{n}}}},$ $K e^{-(s_k-s_{k+1})(\nu+i{\bf{n}}\cdot v)}K$ and $e^{-(s_k-s_{k+1})(\nu+i{\bf{n}}\cdot v)}$, where $k\in \{1,2,\cdots,13\}$ (we use the convention that $s_{14}=0$).
Applying the bounds in ~\eqref{eq:exactForm}, Lemma ~\ref{LM:roughEst} and Lemma ~\ref{LM:oscillate}, we see that there is a constant $C_0>0$ such that
\begin{equation*}
\begin{array}{lll}
& &\|\tilde{A}_{{\bf{n}}}g_{\bf{n}}\|_{L^{1}(\mathbb{R}^{3})}\\
& &\\
&\lesssim&  e^{-C_{0}t} (1+|{\bf{n}}|)\|\langle v\rangle^{20} g_{\bf{n}}\|_{L^{1}(\mathbb{R}^{3})} \times \\
& &\\
& &\int_{0}^{t}\int_{0}^{s_1}\cdots \int_{0}^{s_{12}} [1+|{\bf{n}}|(s_{12}-s_{13})]^{-1} [1+|{\bf{n}}|(s_{8}-s_{11})]^{-1}\cdots [1+|{\bf{n}}|(s_{2}-s_{3})]^{-1}\ ds_{13} ds_{12}\cdots ds_{1}.
\end{array}
\end{equation*}
By direct computation we find that there exists a positive constant $\tilde{C}_0\leq C_0$ such that
\begin{equation*}
\|\tilde{A}_{{\bf{n}}}g_{\bf{n}}\|_{L^{1}(\mathbb{R}^{3})}\lesssim  e^{-\tilde{C}_0t}\frac{1}{(1+|{\bf{n}}|)^{4}} \|\langle v\rangle^{20} g_{\bf{n}}\|_{L^{1}(\mathbb{R}^{3})}.
\end{equation*} Plugging this and ~\eqref{eq:TildeA0} into ~\eqref{eq:decomA}, we find that
\begin{equation}\label{eq:Aprelimi}
\|(1-P_0)\tilde{A}g\|_{L^{1}(\mathbb{R}^{3}\times\mathbb{T}^{3})} \lesssim e^{-\tilde{C}_{0}t} \sum_{{\bf{n}}\in \mathbb{Z}^{3}} \frac{1}{(1+|{\bf{n}}|)^{4}} \|\langle v\rangle^{20}g_{{\bf{n}}}\|_{L^{1}(\mathbb{R}^{3})}.
\end{equation} The terms on the right hand side are bounded by
\begin{equation*}
\|\langle v\rangle^{20}g_{{\bf{n}}}\|_{L^{1}(\mathbb{R}^{3})}\leq (2\pi)^{3} \|\langle v\rangle^{20}g\|_{L^{1}(\mathbb{R}^{3}\times \mathbb{T}^{3})}.
\end{equation*} This, together with the fact that $\displaystyle\sum_{{\bf{n}}\in \mathbb{Z}^{3}} \frac{1}{(1+|{\bf{n}}|)^{4}}<\infty,$ implies that
\begin{equation}\label{eq:completeA}
\|(1-P_0)\tilde{A}g\|_{L^{1}(\mathbb{R}^{3}\times\mathbb{T}^{3})} \lesssim e^{-\tilde{C}_0 t} \|\langle v\rangle^{20}g\|_{L^{1}(\mathbb{R}^{3}\times\mathbb{T}^{3})}.
\end{equation}

Obviously Equation ~\eqref{eq:durha}, Inequality ~\eqref{eq:completeA} and Proposition ~\ref{prop:easyEst} imply Theorem ~\ref{THM:propagator}.
\begin{flushright}
$\square$
\end{flushright}
\subsection{Proof of Proposition ~\ref{prop:easyEst}}\label{subsec:ProofEEst}
Recall the meaning of the constant $\Lambda$ in ~\eqref{eq:globalLower}.
The definition of $A_0$ (see ~\eqref{eq:difA0}) implies that
\begin{equation}\label{eq:a0t}
\|A_{0}(t)f\|_{L^{1}(\mathbb{R}^{3}\times \mathbb{T}^{3})}\leq e^{-\Lambda t}\|f\|_{L^{1}(\mathbb{R}^{3}\times \mathbb{T}^{3})}.
\end{equation}

For $A_1,$ we use the estimate for the unbounded operator $K$ given in Lemma ~\ref{LM:EstNonline}. A direct computation then yields
$$
\begin{array}{lll}
\|A_{1}(f)\|_{L^{1}(\mathbb{R}^{3}\times \mathbb{T}^{3})}&\leq &\int_{0}^{t} e^{-\Lambda(t-s)}\|K e^{-s(\nu+v\cdot\nabla_{x})x}f\|_{L^{1}(\mathbb{R}^{3}\times \mathbb{T}^{3})}\ ds\\
& &\\
&\lesssim& \int_{0}^{t} e^{-\Lambda(t-s)} e^{-\Lambda s}\ ds \|\langle v\rangle f\|_{L^{1}}\\
& &\\
&= &e^{-\Lambda t}t \|\langle v\rangle f\|_{L^{1}}.
\end{array}
$$

Similar arguments yield the desired estimates for $A_{k},\ k=2,3,\cdots 12$.

Thus, the proof of Proposition ~\ref{prop:easyEst} is complete.

\begin{flushright}
$\square$
\end{flushright}
\subsection{Proof of Lemma ~\ref{LM:roughEst}}\label{subsec:RoughEst}
\begin{proof}
If ${\bf{n}}=(0,0,0)$ then the proof of ~\eqref{eq:est000} is similar to that of a similar estimate in ~\cite{Ark1988,Wenn1993,Mouhot2006} and to the proof of ~\eqref{eq:estNot0} given below. It is therefore omitted. What makes the present situation different to the one considered in ~\cite{Ark1988,Wenn1993,Mouhot2006} is that the spectrum of the linear operator $L_{{\bf{n}}}$ depends on ${\bf{n}}$ in a non-trivial manner. The union over ${\bf{n}}$ of the spectra of the operators $L_{\bf{n}}$ fills almost the entire right half of the complex plane.
\begin{figure}[htb!]
\centering%
\includegraphics[width=8cm]{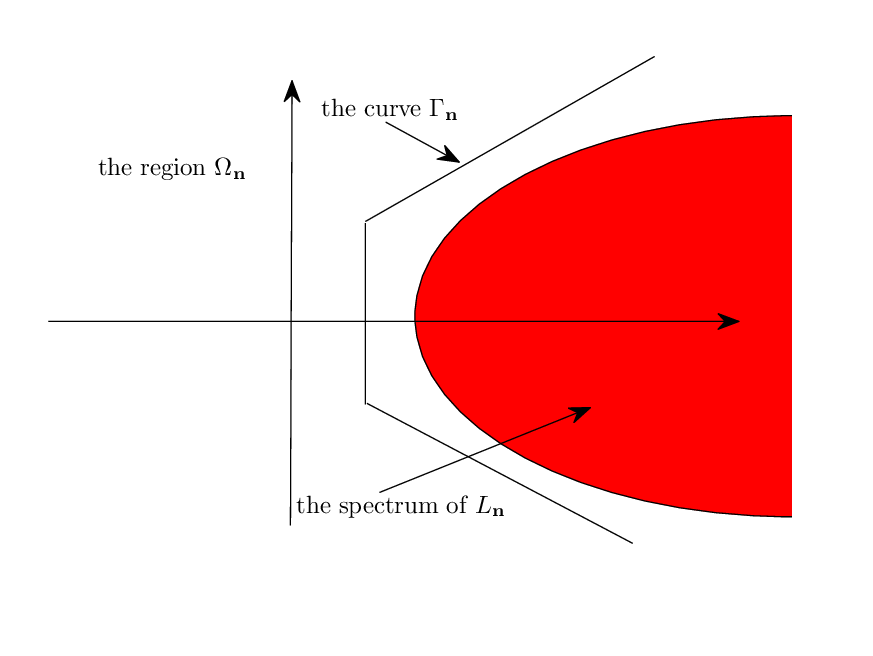}
\caption{The spectrum of $L_{{\bf{n}}}$, the curve $\Gamma_{{\bf{n}}},$ and the region $\Omega_{{\bf{n}}}$}
\label{fig:Ln}
\end{figure}

For any ${\bf{n}}\in \mathbb{Z}^{3},$ we define a curve $\Gamma_{{\bf{n}}}$ (see Figure ~\ref{fig:Ln}),
\begin{equation}\label{eq:difCurve}
\Gamma_{{\bf{n}}}:=\Gamma_{1}({\bf{n}})\cup \Gamma_{2}({\bf{n}})\cup \Gamma_{3}({\bf{n}})
\end{equation} with
$$\Gamma_{1}({\bf{n}}):=\{\Theta+i\beta |\ \beta \in [-\Psi(|{\bf{n}}|+1),\ \Psi(|{\bf{n}}|+1)]\};$$
$$\Gamma_{2}({\bf{n}}):=\{\Theta+i(|{\bf{n}}|+1)\Psi+\beta+i \Psi \beta(|{\bf{n}}|+1) ,\ \beta\geq 0\};$$
$$\Gamma_{3}({\bf{n}}):=\{\Theta-i(|{\bf{n}}|+1)\Psi+\beta-i \Psi \beta(|{\bf{n}}|+1) ,\ \beta\geq 0\}.$$ Here $\Theta$ and $\Psi$ are positive constants to be chosen later; they are independent of the constant $\kappa$ in ~\eqref{eq:NLBL1}.

Moreover, we define $\Omega_{{\bf{n}}}$ to be the complement of the region encircled by the curve $\Gamma_{{\bf{n}}};$ see Figure ~\ref{fig:Ln}.

The following lemma provides an important estimate.
\begin{lemma}\label{LM:distanceContour}
Suppose that the positive constants $\Theta$ and $\frac{1}{\Psi}$ are chosen sufficiently small.
Then there exists a constant $C$ independent of ${\bf{n}}$ such that, for any point $\zeta\in \Omega_{{\bf{n}}}$ and ${\bf{n}}\in \mathbb{Z}^{3}\backslash\{(0,0,0)\},$ we have $$\|(L_{{\bf{n}}}-\zeta)^{-1}\|_{L^{1}\rightarrow L^{1}}\leq C.$$
\end{lemma}
This lemma is proven in Appendix ~\ref{sec:distanceContour}.

This lemma and the spectral theorem in ~\cite{RSI} yield the formula
\begin{equation}\label{eq:IntContour}
e^{-tL_{{\bf{n}}}}=\frac{1}{2\pi i}\oint_{\Gamma_{{\bf{n}}}} e^{-t\zeta}[\zeta-L_{{\bf{n}}}]^{-1} \ d\zeta
\end{equation} on $L^{1}(\mathbb{R}^3).$
Applying Lemma ~\ref{LM:distanceContour} to ~\eqref{eq:IntContour} we obtain that
$$
\|e^{-tL_{{\bf{n}}}}\|_{L^1\rightarrow L^1}\lesssim  \int_{\zeta\in \Gamma_1({\bf{n}})\cup \Gamma_{2}({\bf{n}})\cup \Gamma_{3}({\bf{n}})} e^{-t Re\ \zeta} \ |d\zeta|
$$ By the definition of $\Gamma_{1}({\bf{n}})$, it is easy to see that $$\int_{\zeta\in \Gamma_{1}} e^{-\Theta t} |d\zeta|\lesssim e^{-\Theta t} |{\bf{n}}|.$$
Similarly, the definitions of $\Gamma_{2}({\bf{n}})$ and $\Gamma_{3}({\bf{n}})$ imply that for any $t\geq 1,$  $$\int_{\zeta\in \Gamma_{2}({\bf{n}})\cup \Gamma_{3}({\bf{n}})} e^{-t Re \zeta} \ |d\zeta| \lesssim (1+|{\bf{n}}|)\int_{\Theta}^{\infty} e^{-t\sigma} d\sigma\lesssim e^{-\Theta t} (1+|{\bf{n}}|).$$

Collecting the estimates above, we arrive at ~\eqref{eq:est000}, provided that $t\geq 1.$

The proof will be complete if we can show that the propagator $e^{-tL_{{\bf{n}}}}$ is bounded on $L^{1}(\mathbb{R}^3)$ when $t\in[0,1].$ To prove this, we establish the local wellposedness of the equation
$$
\begin{array}{lll}
\partial_{t}g&=&[-\nu-i{\bf{n}}\cdot v+K]g,\\
g(v,0)&=&g_0(v).
\end{array}
$$ This is easier to prove than local wellposedness of the nonlinear equation in Proposition ~\ref{Prop:wellposed}, and we permit ourselves to omit the details.

This completes the proof of Lemma ~\ref{LM:roughEst}.
\end{proof}
\subsection{Proof of Inequality (~\ref{eq:oscilate})}\label{subsec:oscillate}
\begin{proof}
We denote the integral kernel of the operator $K$ by $K(v,u)$ and infer its explicit form from ~\eqref{eq:difK}, ~\eqref{eq:difK0} and ~\eqref{eq:difK1}.
It is then easy to see that the integral kernel of the operator $Ke^{-t(\nu+i{\bf{n}}\cdot v)}K$ is given by
$$K_{t}^{({\bf{n}})}(v,u):=\int_{\mathbb{R}^{3}}K(v,z) e^{-t [\nu(z)+i{\bf{n}}\cdot z]} K(z,u)\ d^3 z.$$

We use the oscillatory nature of $e^{-it{\bf{n}}\cdot z}$ to derive some ``smallness estimates" when $|{\bf{n}}|$ is sufficiently large. Mathematically, we achieve this by integrating by parts in the variable $z$.
Without loss of generality we assume that $$|n_{1}|\geq \frac{1}{3}|{\bf{n}}|.$$ We then integrate by parts in the variable $z_1$ to obtain
\begin{equation}
\begin{array}{lll}
K_{t}^{({\bf{n}})}(v,u)&=&\int_{\mathbb{R}^{3}} K(v,z)K(z,u) \frac{1}{-t[\partial_{z_1} \nu(z)+in_1]} \partial_{z_1} e^{-t [\nu(z)+i{\bf{n}}\cdot z]} \ d^3 z\\
& &\\
&=& \int_{\mathbb{R}^{3}} \partial_{z_1} [K(v,z)K(z,u) \frac{1}{t[\partial_{z_1} \nu(z)+in_1]} ] e^{-t [\nu(z)+i{\bf{n}}\cdot z]} \ d^3 z
\end{array}
\end{equation}

The different terms in $\partial_{z_1} [K(v,z)K(z,u) \frac{1}{t[\partial_{z_1} \nu(z)+in_1]} ]$ are dealt with as follows.
\begin{itemize}
\item[(1)]
We claim that, for $l=0,1,$
\begin{equation}\label{eq:estKernel}
\int_{\mathbb{R}^{3}}\langle v\rangle^{m}|\partial_{z_{1}}^{l}K(v,z)|\ d^3 v\lesssim \langle z\rangle^{m+2},\ \ \int_{\mathbb{R}^{3}}\langle z\rangle^{m}|\partial_{z_{1}}^{l}K(z,u)|\ d^3 z\lesssim \langle u\rangle^{m+2}.
\end{equation}
\item[(2)] By direct computation,
\begin{equation}
|\partial_{z}^{l}\frac{1}{t[\partial_{z_1} \nu(z)+in_1]}|\lesssim \frac{1}{|{\bf{n}}|t}\ \text{for}\ l=0,1.
\end{equation}
\end{itemize}
These bounds and the fact that $e^{-t\nu}\lesssim e^{-\Lambda t}$ (see ~\eqref{eq:globalLower}) imply that
$$\int_{\mathbb{R}^{3}\times \mathbb{R}^{3}} \langle v\rangle^{m}|K_{t}^{({\bf{n}})}(v,u)g(u)|\ d^3 u\lesssim \frac{e^{-\Lambda t}}{|{\bf{n}}|t}\|\langle v\rangle^{m+3}g\|_{L^{1}}.$$ To remove the non-integrable singularity in the upper bound at $t=0$, we use a straightforward estimate derived from the definition of $K_{t}^{({\bf{n}})}$ to obtain
$$\int_{\mathbb{R}^{3}\times \mathbb{R}^{3}} \langle v\rangle^{m}|K_{t}^{({\bf{n}})}(v,u)g(u)|\ d^3 u\lesssim e^{-\Lambda t}\|\langle v\rangle^{m+3}g\|_{L^{1}}.
$$ Combination of these two estimates yields ~\eqref{eq:oscilate}.

We are left with proving ~\eqref{eq:estKernel}. In the next we focus on proving ~\eqref{eq:estKernel} when $l=1$, the case $l=0$ is easier, hence omitted.
By direct computation we find that $$
\begin{array}{lll}
|\partial_{z_{1}}K(v,z)|&\lesssim & \kappa |\partial_{z_{1}}|z-v|^{-1}e^{-\frac{|(z-v)\cdot v|^2}{|z-v|^2}}|+\kappa|\partial_{z_{1}}|z-v|e^{-|v|^2}| +|\partial_{z_{1}} r_{0}(z,v) |
\end{array}
$$
and, similarly, that $$|\partial_{z_{1}}K(z,u)|\lesssim \kappa |\partial_{z_{1}}|z-u|^{-1}e^{-\frac{|(z-u)\cdot z|^2}{|z-u|^2}}|+\kappa |\partial_{z_{1}}|z-u|e^{-|z|^2}|+|\partial_{z_{1}} r_{0}(u,z) |.$$
Among the various terms we only study the most difficult one, namely $\partial_{z_{1}}K_{1,1}(v,z)$, where $K_{1,1}(v,z)$ is defined by $$K_{1,1}(v,z):=|z-v|^{-1}e^{-\frac{|(z-v)\cdot v|^2}{|z-v|^2}}.$$
By direct computation $$|\partial_{z_{1}}K_{1,1}(v,z)|\lesssim \frac{1+|v_1|}{|v-z|^{2}}  e^{-\frac{1}{2}\frac{|(z-v)\cdot v|^2}{|z-v|^2}}.$$
To complete our estimate we divide the set $(v,z)\in \mathbb{R}^{3}\times \mathbb{R}^{3}$ into two subsets defined by $|v|\leq 10|z|$ and $|v|>10|z|,$ respectively. In the first subset we have that
$$|\partial_{z_{1}}K_{1,1}(v,z)|\lesssim \frac{1}{|v-z|^2}(|v|+1)\leq \frac{10(|z|+1)}{|v-z|^2},$$ and hence
\begin{equation}
\int_{|v|\leq 8|z|}\langle v\rangle^{m}|\partial_{z_{1}}K_{1,1}(v,z)|\ d^3 v\leq 10(1+|z|)^{m+1} \int_{|v|\leq 10|z|}\frac{1}{|v-z|^2}\ d^3 v\lesssim (1+|z|)^{m+2}.
\end{equation}
In the second subset we have that $z-v\approx -v$, which implies that $\frac{|(z-v)\cdot v|}{|z-v|}\geq \frac{1}{2}|v|$. Thus, $$|\partial_{z_{1}}K_{1,1}(v,z)|\leq \frac{1+|v|}{|v|^2}e^{-\frac{1}{8}|v|^2}.$$ This obviously implies that
\begin{equation}
\int_{|v|\geq 10|z|}\langle v\rangle^{m}|\partial_{z_{1}}K_{1,1}(v,z)|\ d^3 v\lesssim \int_{|v|\geq 10|z|}\langle v\rangle^{m}\frac{1+|v|}{|v|^2}\ e^{-\frac{1}{8}|v|^2} d^3 v\lesssim 1.
\end{equation}

By such estimates the proof of ~\eqref{eq:estKernel} can be easily completed.
\end{proof}

\section{Proof of the Main Theorem}\label{sec:ProofMainTHM}
To simplify notations, we let $L^{1}$ stand for $L^{1}(\mathbb{R}^{3}\times \mathbb{T}^{3})$.

Given a solution, $f(\cdot,s)$, $0\leq s\leq t$, of the Boltzmann equation ~\eqref{eq:ALinear}, we introduce two ``control functions", $\mathcal{M}$ and $\mathcal{I}$:
\begin{equation}\label{eq:difCont}
 \begin{array}{lll}
  \mathcal{M}(t)&:=&\displaystyle\max_{0\leq s\leq t} e^{C_{0}s} \sum_{|\alpha|\leq 8} \|\partial_{x}^{\alpha}f(s)\|_{L^{1}},\\
& &\\
\mathcal{I}(t)&:=&\displaystyle\sum_{|\alpha|\leq 8}\int_{0}^{t} \|\langle v\rangle^{2m+2}\partial_{x}^{\alpha}f(s)\|_{L^{1}}\ ds,
 \end{array}
\end{equation} where, the constants $m$ and $C_0$ are as in Theorem ~\ref{THM:propagator}.

These two functions can be estimated as follows.
\begin{lemma}\label{LM:Majorants}
\begin{equation}\label{eq:controlM}
 \mathcal{M}(t)\leq C[\sum_{|\alpha|\leq 8}\|\langle v\rangle^{m}\partial_{x}^{\alpha}f_0\|_{L^{1}}+ \kappa \mathcal{M}^{\frac{3}{2}}\mathcal{I}^{\frac{1}{2}}];
\end{equation} and
\begin{equation}\label{eq:controlI}
\mathcal{I}(t) \leq C[ \sum_{|\alpha|\leq 8} \|\langle v\rangle^{2m+1} \partial_{x}^{\alpha}f_{0}\|_{L^{1}}+\kappa \mathcal{I}(t)\mathcal{M}(t)+\mathcal{M}(t)].
\end{equation} for a finite constant $C$, where $f_0$ is the initial condition.
\end{lemma}
This lemma will be proven below.

We are now ready to prove our main result, Theorem ~\ref{THM:MainTHM}.\\
{\bf{Proof of Theorem ~\ref{THM:MainTHM}}} By local wellposedness of the equation, there exists a time interval $[0,T]$, $T=T(\kappa),$ such that
\begin{equation}\label{eq:assum}
\mathcal{M}(t)\leq \kappa^{-\frac{1}{4}},\ \text{for any time} \ t\in [0,T].
\end{equation}
We move the term $C\kappa \mathcal{I}(t)\mathcal{M}(t)$ on the right hand side of ~\eqref{eq:controlI} to the left hand side and then use the fact that $C\kappa \mathcal{M}(t)\leq \frac{1}{2}$ to conclude that
\begin{equation}\label{eq:prelinI}
\mathcal{I}(t) \leq 2C[\sum_{|\alpha|\leq 8}\|\langle v\rangle^{2m}\partial_{x}^{\alpha}f_{0}\|_{L^{1}}+\mathcal{M}(t)].
\end{equation}
Plugging this bound into the right hand side of ~\eqref{eq:controlM} and using ~\eqref{eq:assum}, we obtain that
$$\mathcal{M}(t)\leq C (\sum_{|\alpha|\leq 8}\|\langle v\rangle^{m}\partial_{x}^{\alpha}f_0\|_{L^{1}}+\kappa^{\frac{1}{2}}\mathcal{M}(t)).$$
This, together with the fact $C\kappa^{\frac{1}{2}}\leq \frac{1}{2},$ implies that
\begin{equation}\label{eq:prelinM}
\mathcal{M}(t)\leq 2C \sum_{|\alpha|\leq 8}\|\langle v\rangle^{m}\partial_{x}^{\alpha}f_0\|_{L^{1}}, \ \text{for any time}\ t\in [0,T].
\end{equation}
This in turn implies that ~\eqref{eq:assum} holds on a larger time interval. By running the arguments ~\eqref{eq:assum}-~\eqref{eq:prelinM} iteratively we find that ~\eqref{eq:prelinM} holds on the time interval $[0,\infty).$

Using the definition of $\mathcal{M}$, in ~\eqref{eq:difCont}, we obtain that, for any time $t\in [0,\infty),$
\begin{equation}
\sum_{|\alpha|\leq 8}\|\partial_{x}^{\alpha}f(t)\|_{L^{1}}\leq 2C e^{-C_0 t} \sum_{|\alpha|\leq 8}\|\langle v\rangle^{m}\partial_{x}^{\alpha}f_0\|_{L^{1}},
\end{equation}
which together with the definition of $f$, see ~\eqref{eq:difF}, implies inequality ~\eqref{eq:expon} in Theorem ~\ref{THM:MainTHM}.

The proof of Theorem ~\ref{THM:MainTHM} is complete.
\begin{flushright}
$\square$
\end{flushright}
\subsection{Proof of Lemma ~\ref{LM:Majorants}}
\begin{proof}
We apply Duhamel's principle to rewrite the Boltzmann equation ~\eqref{eq:ALinear} as
$$
f=e^{-tL}(1-P_0)f_0+\kappa\int_{0}^{t}e^{-(t-s)L}(1-P_0)Q(f,f)(s)\ ds,
$$ Here, the fact that $(1-P_0)f=f$, which is implied by ~\eqref{eq:difF} and the definition of $P_0$ in ~\eqref{eq:difProjection}, has been used. We apply the propagator estimate in Theorem ~\ref{THM:propagator} to conclude that, for any $\alpha\in (\mathbb{Z}^{+})^{3}$ with $|\alpha|\leq 8,$
\begin{equation}\label{eq:estPre}
\begin{array}{lll}
\|\partial_{x}^{\alpha}f(\cdot,t)\|_{L^{1}}&\lesssim & e^{-C_{0}t}\|\langle v\rangle^{m}\partial_{x}^{\alpha}f_0\|_{L^{1}}+\kappa \int_{0}^{t} e^{-C_{0}(t-s)}
\|\langle v\rangle^{m}\partial_{x}^{\alpha}Q(f,f)(s)\|_{L^{1}}\ ds
\end{array}
\end{equation}
To estimate the nonlinear term on the right hand side, we use techniques similar to those in ~\eqref{eq:FirstStep} to obtain
$$
\begin{array}{lll}
\|\langle v\rangle^{m}\partial_{x}^{\alpha}Q(f,f)\|_{L^{1}}&\lesssim &\displaystyle\sum_{l=0}^{m} \sum_{|\beta_1|\leq 8}\|\langle v\rangle^{k} \partial_{x}^{\beta_1}f\|_{L{1}}\sum_{|\beta_2|\leq 8}\|\langle v\rangle^{m+1-k}\partial_{x}^{\beta_2}f\|_{L^{1}}\\
& &\\
&\lesssim &\displaystyle\sum_{|\beta_1|,\ |\beta_2|\leq 8}\|\langle v\rangle^{m+1}\partial_{x}^{\beta_1}f\|_{L^{1}} \|\partial_{x}^{\beta_2}f\|_{L^{1}}.
\end{array}
$$
To the term $\|\langle v\rangle^{m}\partial_{x}^{\beta_1}f\|_{L^{1}}$ on the right hand side we apply the Schwarz inequality to obtain
$$\|\langle v\rangle^{m+1}\partial_{x}^{\beta}f\|_{L^{1}}^2 \leq \|\langle v\rangle^{2m+2}\partial_{x}^{\beta}f\|_{L^{1}}\|\partial_{x}^{\beta}f\|_{L^{1}}.$$ Plugging this into ~\eqref{eq:estPre}, we obtain that
$$
\|\partial_{x}^{\alpha}f\|_{L^{1}}\lesssim  e^{-C_0 t}\|\langle v\rangle^{m}\partial_{x}^{\alpha}f_0\|_{L^{1}}+\displaystyle\kappa\sum_{|\beta_1|,\ |\beta_2|\leq 8} \int_{0}^{t} e^{-C_0(t-s)}
\|\langle v\rangle^{2m+2}\partial_{x}^{\beta_1}f\|_{L^{1}}^{\frac{1}{2}} \|\partial_{x}^{\beta_2}f\|_{L^{1}}^{\frac{3}{2}}\ ds$$
Applying the Schwarz inequality again and using the definitions of $\mathcal{M}$ and $\mathcal{I},$ we find that $$
\begin{array}{lll}
\|\partial_{x}^{\alpha}f\|_{L^{1}}
&\lesssim & e^{-C_0 t}\|\langle v\rangle^{m}\partial_{x}^{\alpha}f_0\|_{L^{1}}+\displaystyle\sum_{|\beta_1|,\ |\beta_2|\leq 8}\kappa [\int_{0}^{t} e^{-2C_0(t-s)}
 \|\partial_{x}^{\beta_1}f\|_{L^{1}}^{3}\ ds]^{\frac{1}{2}} [\int_{0}^{t}\|\langle v\rangle^{2m+2}\partial_{x}^{\beta_2}f\|_{L^{1}} \ ds]^{\frac{1}{2}}\\
& &\\
&\leq & e^{-C_0 t}\|\langle v\rangle^{m}\partial_{x}^{\alpha}f_0\|_{L^{1}} +\kappa e^{-C_{0}t} \mathcal{M}^{\frac{3}{2}}\mathcal{I}^{\frac{1}{2}}.
\end{array}
$$
Recalling the definition of $\mathcal{M},$ we see that the proof of ~\eqref{eq:controlM} is complete.

To prove ~\eqref{eq:controlI}, or to estimate $\int_{0}^{t} \|\langle v\rangle^{2m+2} f(s)\|_{L^{1}} \ ds,$ we rewrite ~\eqref{eq:ALinear} as
$$\partial_{x}^{\alpha}f(t)= e^{-t(\nu+v\cdot \nabla_{x})}\partial_{x}^{\alpha}f_0+\int_{0}^{t} e^{-(t-s)(\nu+v\cdot\nabla_{x})} \partial_{x}^{\alpha} H(s)\ ds, $$ where $H(s)$ is defined by
$$H(s):=K_{0}f(s)+\kappa K_{1}f(s)+\kappa Q(f,f)(s).$$ By direct computation and the fact that $L^{1}$-norm is preserved under the mapping $e^{-tv\cdot \nabla_{x}}$ we obtain
$$
\begin{array}{lll}
\|\langle v\rangle^{2m+2}\partial_{x}^{\alpha}f(t)\|_{L^{1}}&\leq &\|\langle v\rangle^{2m+2}e^{-t[\nu+v\cdot\nabla_{x}]}\partial_{x}^{\alpha}f_{0}\|_{L^{1}} +\int_{0}^{t}\|\langle v\rangle^{2m+2}e^{-(t-s)[\nu+v\cdot \nabla_{x}]}\partial_{x}^{\alpha}H(s)\|_{L^{1}}\ ds\\
& &\\
&=&\|\langle v\rangle^{2m+2}e^{-t\nu}\partial_{x}^{\alpha}f_{0}\|_{L^{1}} +\int_{0}^{t}\|\langle v\rangle^{2m+2}e^{-(t-s)\nu}\partial_{x}^{\alpha}H(s)\|_{L^{1}}\ ds
\end{array}
$$
Integrate both sides from $0$ to $t$, and use the obvious fact that $\int_{0}^{t}\|g(s)\|_{L^{1}} \ ds=\|\int_{0}^{t}|g|(s)\ ds\|_{L^{1}}$ we arrive at
\begin{equation*}
\int_{0}^{t} \|\langle v\rangle^{2m+2}\partial_{x}^{\alpha}f(s)\|_{L^{1}}\ ds \leq \|\int_{0}^{t}\langle v\rangle^{2m+2}e^{-s\nu}|\partial_{x}^{\alpha}f_{0}| \ ds\|_{L^{1}}+\|\int_{0}^{t}\int_{0}^{s} \langle v\rangle^{2m+2} e^{-(s-s_1)\nu}|\partial_{x}^{\alpha}H(s_1)|\ ds_1ds\|_{L^{1}}.
\end{equation*}
The first term on the right hand side can be integrated explicitly. For the second term, we integrate by parts in the variable $s$. We find that
\begin{equation}\label{eq:inteByParts}
\begin{array}{lll}
\int_{0}^{t} \|\langle v\rangle^{2m+2}\partial_{x}^{\alpha}f(s)\|_{L^{1}}\ ds&\leq &\|\nu^{-1}\langle v\rangle^{2m+2} \partial_{x}^{\alpha}f_0\|_{L^{1}}+\int_{0}^{t} \|\nu^{-1}\langle v\rangle^{2m+2} \partial_{x}^{\alpha}H(s)\|_{L^{1}}\ ds\\
& &\\
&\lesssim &\|\langle v\rangle^{2m+1} \partial_{x}^{\alpha}f_0\|_{L^{1}}+\int_{0}^{t} \|\langle v\rangle^{2m+1} \partial_{x}^{\alpha}H(s)\|_{L^{1}}\ ds.
\end{array}
\end{equation} In the last step we use the estimate for $\nu=\nu_0+C_{\infty}\kappa \nu_1$ in ~\eqref{eq:BdBelow}, which, through its definition, makes it necessary to require that $r_0$ in ~\eqref{eq:NLBL1} be unbounded.

This together with the definition of $H$, the estimates on $K_0$ and $K_1$ in Lemma ~\ref{LM:EstNonline} and on the nonlinearity in ~\eqref{eq:AppliEmbed}, implies that there exist constants $C_1$ and $C_2$ such that
\begin{equation}\label{eq:tedious}
\begin{array}{lll}
& &\int_{0}^{t} \|\langle v\rangle^{2m+2}\partial_{x}^{\alpha}f(s)\|_{L^{1}}\ ds\\
& &\\
&\leq &C_1[\|\langle v\rangle^{2m+1} \partial_{x}^{\alpha}f_0\|_{L^{1}}+
C_{\infty}\kappa \int_{0}^{t} \|\langle v\rangle^{2m+1}K_1\partial_{x}^{\alpha} f(s)\|_{L^{1}}\ ds \\
& &\\
& &+\int_{0}^{t} \|\langle v\rangle^{2m+1}K_0\partial_{x}^{\alpha} f(s)\|_{L^{1}}\ ds+\kappa\int_{0}^{t} \|\langle v\rangle^{2m+1}\partial_{x}^{\alpha}Q(f,f)(s)\|_{L^{1}}\ ds\\
& &\\
&\leq & C_2[\|\langle v\rangle^{2m+1} \partial_{x}^{\alpha}f_0\|_{L^{1}}+\kappa C_{\infty} \int_{0}^{t} \|\langle v\rangle^{2m+2}\partial_{x}^{\alpha} f(s)\|_{L^{1}}\ ds\\
& &\\
& &+\int_{0}^{t} \|\langle v\rangle\partial_{x}^{\alpha} f(s)\|_{L^{1}}\ ds+\kappa \displaystyle\sum_{|\beta_1|,\ |\beta_2|\leq 8}\int_{0}^{t} \|\langle v\rangle^{2m+2}\partial_{x}^{\beta_1}f(s)\|_{L^{1}} \|\partial_{x}^{\beta_2}f(s)\|_{L^{1}}\ ds].
\end{array}
\end{equation}
We use the Schwarz inequality to estimate the third term, $\int_{0}^{t} \|\langle v\rangle \partial_{x}^{\alpha}f(s)\|_{L^{1}}\ ds,$ on the right hand side:
$$
C_2\int_{0}^{t} \|\langle v\rangle\partial_{x}^{\alpha} f(s)\|_{L^{1}}\ ds\leq \frac{1}{2}\int_{0}^{t} \|\langle v\rangle^{2m+2} \partial_{x}^{\alpha}f(s)\|_{L^{1}}\ ds+C_3(m)\int_{0}^{t} \| \partial_{x}^{\alpha}f(s)\|_{L^{1}}\ ds,
$$ for any $m\geq 0$ with $C_{3}(m)\geq 0$.

Inserting this in ~\eqref{eq:tedious} and using that $\kappa>0$ is a small constant, we find that
$$\int_{0}^{t}\|\langle v\rangle^{2m+2}\partial_{x}^{\alpha}f(s)\|_{L^{1}}\ ds\lesssim \|\langle v\rangle^{2m+1}\partial_{x}^{\alpha}f_0\|_{L^{1}}
+\kappa \displaystyle\sum_{|\beta|\leq 8}\int_{0}^{t} \|\langle v\rangle^{2m+2} \partial_{x}^{\beta}f(s)\|_{L^{1}}\ ds\mathcal{M}(t)+\mathcal{M}(t),$$ which together with the definition of $\mathcal{I}$ in ~\eqref{eq:difCont} implies the desired estimate ~\eqref{eq:controlI}.
\end{proof}

\appendix
\section{Proof of Lemma ~\ref{LM:EstNonline}}\label{sec:nonlinear}
It is easy to derive ~\eqref{eq:globalLower} and ~\eqref{eq:estK0} by the definitions of $\nu_0,\ \nu_1$ and $K_0$. We therefore omit the details.

We start with ~\eqref{eq:estNonL}. By direct computation
\begin{equation}\label{eq:Angular}
\begin{array}{lll}
\|\langle v\rangle^{m}Q(f,g)\|_{L^{1}(\mathbb{R}^3)}
&=& \int_{\mathbb{R}^{3}\times \mathbb{R}^{3}\times \mathbb{S}^{2}} \langle v\rangle^{m} |(u-v)\cdot \omega| f(u') |g|(v') \ d^3 u d^3 v d^2 \omega\\
& &\\
& &+\int_{\mathbb{R}^{3}\times \mathbb{R}^{3}\times \mathbb{S}^{2}} \langle v\rangle^{m} |(u-v)\cdot \omega| f(u) |g|(v) \ d^3 u d^3 v d^2 \omega.
\end{array}
\end{equation}

It is easy to control the second term on the right hand side.

We then turn to the first term. For any fixed $\omega\in \mathbb{S}^{2},$ the mapping from $(u,v)\in \mathbb{R}^{6}$ to $(u',v')\in \mathbb{R}^{6}$ is a linear symplectic transformation, hence
\begin{equation}
d^3u d^3 v=d^3 u' d^3 v'
\end{equation} where, $u'$ and $v'$ are defined ~\eqref{eq:NLBL1}. This together with the observation that
\begin{equation}
\langle v\rangle^{m}\lesssim \langle u'\rangle^{m}+\langle v'\rangle^{m},\ \text{and}\ |(u-v)\cdot \omega|\lesssim |u'|+|v'|
\end{equation} and ~\eqref{eq:Angular} obviously implies ~\eqref{eq:estNonL}.

As one can infer from the definition $K_1$ in ~\eqref{eq:ALinear}, ~\eqref{eq:mK1} is a special case of ~\eqref{eq:estNonL} by setting $f$ or $g$ to be $M=e^{-|v|^2}.$
\begin{flushright}
$\square$
\end{flushright}
\section{Proof of Lemma ~\ref{LM:distanceContour}}\label{sec:distanceContour}
We start by simplifying the problem. Using the definitions of the operators $L_{{\bf{n}}}$, ${\bf{n}}\in \mathbb{Z}^3$, in ~\eqref{eq:observation}, $K$ in ~\eqref{eq:difK}, and $\nu$ in ~\eqref{eq:difNu} we find that $$L_{{\bf{n}}}=\nu_0-K_0+i{\bf{n}}\cdot v +C_{\infty}\kappa [\nu_1+K_1].$$
The smallness of the constant $\kappa$ suggests to consider $\nu_0-K_0+i{\bf{n}}\cdot v$ as the dominant part. We then convert the estimate on $L_{{\bf{n}}}-\zeta$ to one on $\nu_0-K_0+i{\bf{n}}\cdot v -\zeta$.

To render this idea mathematically rigorous, we show that, in order to prove invertibility of $L_{{\bf{n}}}-\zeta,\ \zeta\in \Omega_{{\bf{n}}},$ it is sufficient to prove this property for $1-K_{\zeta,{\bf{n}}}$, with $K_{\zeta,{\bf{n}}}$ defined by
\begin{equation}\label{eq:difKsn}
K_{\zeta,{\bf{n}}}:=K_0(\nu_0+i{\bf{n}} \cdot v -\zeta)^{-1}.
\end{equation} We rewrite $L_{{\bf{n}}}-\zeta$ as follows:
\begin{equation}\label{eq:rewriteLN}
\begin{array}{lll}
L_{{\bf{n}}}-\zeta&=&[1-K_{\zeta,{\bf{n}}}+C_{\infty}\kappa  (\nu_1+K_1)(\nu_0+i{\bf{n}}\cdot v-\zeta)^{-1}](\nu_0+i{\bf{n}}\cdot v-\zeta)\\
& &\\
&=&(1-K_{\zeta,{\bf{n}}})[1+C_{\infty}\kappa (1-K_{\zeta,{\bf{n}}})^{-1}(\nu_1+K_1)(\nu_0+i{\bf{n}}\cdot v-\zeta)^{-1}](\nu_0+i{\bf{n}}\cdot v-\zeta).
\end{array}
\end{equation} We have the following estimates on the different terms on the right hand side:
\begin{itemize}
\item[(1)]
Concerning $\nu_0+i{\bf{n}}\cdot v-\zeta,$ we observe that it is a multiplication operator. If the constants $\theta$ and $\frac{1}{\Psi}$ in the definition of the curves $\Gamma_{k,{\bf{n}}}, \ k=0,1,2,$ in ~\eqref{eq:difCurve}, are sufficiently small then there exists a constant $C$ such that for any $\zeta\in\Omega_{{\bf{n}}}$
\begin{equation}\label{eq:ReasonCurve}
|\nu_0+i {\bf{n}}\cdot v-\zeta|^{-1} \leq C(1+|v|)^{-1}.
\end{equation}
It is straightforward, but a little tedious to verify this. Details are omitted.
\item[(2)]
Concerning the term $1-K_{\zeta,{\bf{n}}},$ we have the following lemma.
\begin{lemma}\label{LM:compactLm}
Suppose that the constants $\Theta$ and $\frac{1}{\Psi}$ in ~\eqref{eq:difCurve} are sufficiently small.
Then, for any point $\zeta\in \Omega_{{\bf{n}}}$ and ${\bf{n}}\in \mathbb{Z}^{3}\backslash\{(0,0,0)\},$ we have that $1-K_{\zeta,{\bf{n}}}$ is invertible; its inverse satisfies the estimate $$\|(1-K_{\zeta,{\bf{n}}})^{-1}\|_{L^{1}\rightarrow L^{1}}\leq C,$$ where the constant $C$ is independent of ${\bf{n}}$ and $\zeta$.
\end{lemma}
This lemma will be reformulated as Lemmas ~\ref{LM:invertibility} and ~\ref{LM:surjective} below.

\item[(3)]
With ~\eqref{eq:ReasonCurve}, Lemma ~\ref{LM:compactLm} and our estimates on $\nu_1$ and $K_1$ in ~\eqref{eq:globalLower} and ~\eqref{eq:mK1}, we conclude that if $\kappa$ is sufficiently small then the operator $C_{\infty}\kappa (1-K_{\zeta,{\bf{n}}})^{-1}(\nu_1+K_1)(\nu_0+i{\bf{n}}\cdot v-\zeta)^{-1}:\ L^{1}(\mathbb{R}^3)\rightarrow L^1(\mathbb{R}^3)$ in ~\eqref{eq:rewriteLN} is small in norm $\|\cdot\|_{L^1\rightarrow L^1}$. This proves that
\begin{equation}
1+C_{\infty}\kappa (1-K_{\zeta,{\bf{n}}})^{-1}(\nu_1+K_1)(\nu_0+i{\bf{n}}\cdot v-\zeta)^{-1}:\ L^{1}(\mathbb{R}^3)\rightarrow L^1(\mathbb{R}^3)
\end{equation} is invertible.
\end{itemize}
The results above complete the proof of Lemma ~\ref{LM:distanceContour}, assuming that Lemma ~\ref{LM:compactLm} holds.

We divide the proof of Lemma ~\ref{LM:compactLm} into steps. In the first step we prove
\begin{lemma}\label{LM:invertibility}
There exists a constant $C>0$ such that, for any $\zeta\in \Omega_{{\bf{n}}}$ and ${\bf{n}}\in \mathbb{Z}^{3}\backslash\{(0,0,0)\},$
\begin{equation}\label{eq:BdBelow}
\|(1-K_{\zeta,{\bf{n}}}) g\|_{L^{1}}\geq C\|g\|_{L^{1}}.
\end{equation}
\end{lemma}
This will be proven in Subsection ~\ref{subsec:onto} below.

We now present the strategy of the proof of Lemma ~\ref{LM:invertibility}. Our key observation is that the bounded operators $K_{\zeta,{\bf{n}}}$, $\zeta\in \Omega_{{\bf{n}}}$, ${\bf{n}}\in \mathbb{Z}^{3}\backslash\{(0,0,0)\};$ are compact (see Lemma ~\ref{LM:compactness} below). Hence if ~\eqref{eq:BdBelow} does not hold, then there exist some $\zeta_0\in \Omega_{{\bf{n}_0}},$ ${\bf{n}_0}\in \mathbb{Z}^{3}\backslash\{(0,0,0)\}$ and some nontrivial function $g\in L^{1}(\mathbb{R}^3)$ such that $[1-K_{\zeta_0,{\bf{n}_0}}]g=0.$ From the definition of $K_{\zeta_0,{\bf{n}_0}}$ in ~\eqref{eq:difKsn} and the properties of $K_0$ in ~\eqref{eq:difK0} (see also ~\eqref{eq:R0exam}) then we infer that the function $\tilde{g}:=e^{\frac{1}{2}|v|^2} (-\nu_0-i{\bf{n}}_{0}\cdot v+\zeta_0)g$ belongs to $L^{2}(\mathbb{R}^3)$ and satisfies the equation $$(-\nu_0+\tilde{K}_0-i{\bf{n}_0}\cdot v+\zeta_0)\tilde{g}=0.$$ Here $\tilde{K}_{0}:=e^{\frac{1}{2}|v|^2}K_0 e^{-\frac{1}{2}|v|^2}:\ L^{2}(\mathbb{R}^3)\rightarrow L^{2}(\mathbb{R}^3)$ is a self-adjoint and compact operator. By considering spectral properties of $-\nu_0+\tilde{K}_0:\ L^{2}(\mathbb{R}^3)\rightarrow L^{2}(\mathbb{R}^3)$, we exclude the possibility that $\zeta_0\in \Omega_{{\bf{n}_0}}.$ For details we refer the reader to subsection ~\ref{subsec:onto} below.

However, ~\eqref{eq:BdBelow} does not guarantee that the mapping $1-K_{\zeta,{\bf{n}}}$ is onto. To show this we prove, in a second step, the following lemma.
\begin{lemma}\label{LM:surjective} For any $\zeta\in \Omega_{{\bf{n}}},$ and ${\bf{n}}\in \mathbb{Z}^{3}\backslash\{(0,0,0)\},$ the mapping
\begin{equation}
 1-K_{\zeta,{\bf{n}}}: \ L^{1}(\mathbb{R}^3)\rightarrow L^{1}(\mathbb{R}^3)\ \text{is onto.}
\end{equation}
\end{lemma}
This lemma will be proven in Subsection ~\ref{subsec:onto2}.

Lemma ~\ref{LM:invertibility} implies that $1-K_{\zeta,{\bf{n}}}$ maps $L^{1}(\mathbb{R}^3)$ into a closed subset of $L^{1}(\mathbb{R}^3)$. This, together with the `onto-properties' in Lemma ~\ref{LM:surjective}, implies that it is invertible, and its inverse is uniformly bounded. Hence Lemma ~\ref{LM:compactLm} follows.

\subsection{Proof of Lemma ~\ref{LM:invertibility}}\label{subsec:onto}
In what follows we prove ~\eqref{eq:BdBelow} for $\zeta\in \Gamma_0$, the proofs for the other cases are similar.

It is enough to show that there exist constants $C$ and $\Theta$ independent of ${\bf{n}}$ such that, for any $\epsilon\in [0,\Theta]$, $h\in \mathbb{R}$ and ${\bf{n}}\in \mathbb{Z}^3\backslash\{(0,0,0)\},$ we have that
\begin{equation}
\|(1-K_{\epsilon+ih,{\bf{n}}})g\|_{L^{1}}\geq C\|g\|_{L^{1}}.
\end{equation}

Suppose that this inequality does not hold. Then there would exist a sequence $\{\epsilon_{m}\}_{m=1}^{\infty}\subset\mathbb{R}^{+},$ with $\displaystyle\lim_{m\rightarrow \infty}\epsilon_{m}=0$, a sequence $\{h_{m}\}_{m=1}^{\infty}\subset\mathbb{R}$, a sequence $\{g_{m}\}_{m=1}^{\infty}\subset L^{1}(\mathbb{R}^3),$ with $\|g_{m}\|_{L^{1}}=1,$ and a sequence $\{{\bf{n}}_{m}\}\subset \mathbb{Z}^{3}\backslash\{(0,0,0)\}$ such that
\begin{equation}\label{eq:assume}
 \|(1-K_{\epsilon_m+ih_{m},{\bf{n}}_{m}})g_{m}\|_{L^{1}}\rightarrow 0,\ \ \ \text{as}\ m\rightarrow \infty.
\end{equation}
By Lemma ~\ref{LM:compactness} below, the sequence $\{K_{\epsilon_m+ih_{m},{\bf{n}}_{m}}g_{m}\}_{m=1}^{\infty}$ contains a convergent subsequence. Without loss of generality we assume that $$\{K_{\epsilon_m+ih_{m},{\bf{n}}_{m}}g_{m}\}_{m=1}^{\infty}$$ is convergent, i.e. there exists a function $g_{\infty}\in L^{1}$ such that
\begin{equation}\label{eq:compactness}
 \|g_{\infty}+K_{\epsilon_m+ih_{m},{\bf{n}}_{m}}g_{m}\|_{L^{1}}\rightarrow 0,\ \ \text{as}\ m\rightarrow \infty.
\end{equation} This, together with ~\eqref{eq:assume}, implies that
\begin{equation}\label{eq:gInfty}
 \|g_{\infty}-g_{m}\|_{L^{1}}\rightarrow 0,\ \text{as}\ m\rightarrow \infty,\ \text{with}\ g_{\infty}\not=0.
\end{equation}

It is easy to see that the sequences $\{h_{m}\}_{m=1}^{\infty}$ and $\{{\bf{n}}_{m}\}_{m=1}^{\infty}$ are uniformly bounded. Otherwise, by the definition of $K_{\epsilon+ih,{\bf{n}}},$ it is easy to see that $K_{\epsilon+ih,{\bf{n}}}g_{\infty}\rightarrow 0$ as $|h|$ or $|{\bf{n}}|\rightarrow \infty$. This in turn contradicts ~\eqref{eq:assume}.

The bounded sequences $\{h_{m}\}_{m=1}^{\infty}$ and $\{{\bf{n}}_{m}\}_{m=1}^{\infty}$ must contain some convergent subsequences. Without loss of generality, we may assume that there exist a constant $h_{\infty}\in \mathbb{R}$ and ${\bf{n}}_{\infty}\not=(0,0,0)$ such that $h_{\infty}=\displaystyle\lim_{m\rightarrow \infty}h_{m}$ and ${\bf{n}}_{\infty}=\displaystyle\lim_{m\rightarrow \infty}{\bf{n}}_{m}.$ This, together with the definition of $K_{\epsilon+ih,{\bf{n}}}$, implies $$\|K_{\epsilon_m+ih_m,{\bf{n}}_{m}}-K_{ih_{\infty},{\bf{n}}_{\infty}}\|_{L^1\rightarrow L^1}\rightarrow 0,\ \ \text{as}\ m\rightarrow \infty.$$ Using ~\eqref{eq:compactness} and ~\eqref{eq:gInfty}, we conclude that
\begin{equation}\label{eq:noEigen}
 g_{\infty}-K_{ih_{\infty},{\bf{n}}_{\infty}}g_{\infty}=0.
\end{equation}

Recalling the definition of $K_{ih_\infty,{\bf{n}}_{\infty}},$ we find that $|g_{\infty}|\leq Ce^{-\frac{3}{4}|v|^2}$. This enables us to define a function $\tilde{g}_{\infty}\in L^{2}$ by $$\tilde{g}_{\infty}:=e^{\frac{1}{2}|v|^2}(-\nu_0-i{\bf{n}}_{\infty}\cdot v-ih_{\infty})g_{\infty}.$$
By ~\eqref{eq:noEigen}
\begin{equation}\label{eq:fakeState}
(-\nu_0-i{\bf{n}}_{\infty}\cdot v-ih_{\infty}+\tilde{K}_{0})\tilde{g}_{\infty}=0
\end{equation} with $\tilde{K}_{0}:=e^{\frac{1}{2}|v|^2}K_0 e^{-\frac{1}{2}|v|^2}.$

On the other hand, in Lemma ~\ref{LM:lowestEi} below, we prove that $0$ is a simple and the lowest eigenvalue of the self-adjoint operator $-\nu_0+\tilde{K}_{0}: \ L^2\rightarrow L^2,$ with eigenvector $e^{-\frac{1}{2}|v|^2}$. This implies that $\langle g, \ (-\nu_0+\tilde{K}_0)g\rangle=Re\langle g, \ (-\nu_0-i{\bf{n}}_{\infty}\cdot v-ih_{\infty}+\tilde{K}_0)g\rangle=0$ only holds if $g$ is parallel to $e^{-\frac{1}{2}|v|^2}$. By direct computation we find that ~\eqref{eq:fakeState} can not hold if ${\bf{n}}\not=(0,0,0)$, and this completes our proof of Lemma ~\ref{LM:invertibility}.
\begin{flushright}
$\square$
\end{flushright}
The following result has been used in the proof.
\begin{lemma}\label{LM:compactness}
For a sequence $\{g_{m}\}_{m=1}^{\infty}\subset L^{1}(\mathbb{R}^3)$ satisfying $\|g_{m}\|_{L^{1}}\leq 1,$ there exists a subsequence $\{\tilde{g}_{m}\}_{m=1}^{\infty}$ such that $K_{\epsilon+ih, {\bf{n}}}\tilde{g}_{m}$ is convergent in $L^{1}(\mathbb{R}^3),$ i.e. there exists a function $\tilde{g}_{\infty}\in L^{1}(\mathbb{R}^3)$ such that
\begin{equation}
\|\tilde{g}_{\infty}-K_{\epsilon+ih, {\bf{n}}}\tilde{g}_{m}\|_{L^{1}}=0,\ \  \text{as}\ m\rightarrow \infty.
\end{equation}
\end{lemma}
\begin{proof}
This result is a simple generalization of Ascoli's Theorem in ~\cite{RSI} which asserts compactness of any sequence of equi-continuous $L^1$ functions defined in a bounded domain. In the present situation we observe that
\begin{itemize}
\item[(1)]
the sequence of functions $\{K_{\epsilon+ih, {\bf{n}}}\tilde{g}_{m}\}_{m=1}^{\infty}$ is equicontinuous;
\item[(2)] these functions are ``almost compactly supported," in the sense that the functions $e^{\frac{1}{2}|v|^2}K_{\epsilon+ih, {\bf{n}}}\tilde{g}_{m}$, $m=1,2,\cdots,$ are in $L^{1}(\mathbb{R}^3)$ and their norms are uniformly bounded.
\end{itemize}
\end{proof}
\subsection{Proof of Lemma ~\ref{LM:surjective}}\label{subsec:onto2}
To simplify notation we denote $K_{\zeta,{\bf{n}}}$ by $\Phi$, i.e., $$\Phi= K_{\zeta,{\bf{n}}}.$$ This will not cause confusion, because $\zeta$ and ${\bf{n}}$ are fixed in the present subsection.

We start by considering a family of operators $\{1-\delta \Phi|\  \delta\in [0,1]\}.$ The first result is
\begin{lemma}\label{LM:boundBelow}
The operator $1-\delta \Phi$ is bounded there exists a constant $C$ independent of $\delta$ such that
\begin{equation}
 \|1-\delta \Phi\|_{L^{1}\rightarrow L^{1}}\geq C.
\end{equation}
\end{lemma}
\begin{proof}
The important observation is that the operator $-\nu_0-i{\bf{n}}\cdot v+\delta K_{0}$ does not have any purely imaginary or $0$ eigenvalues when $\delta\in [0,1].$ The completion of the proof is similar to the proof of Lemma ~\ref{LM:invertibility}.
\end{proof}

Lemma ~\ref{LM:boundBelow} implies that $1-\delta \Phi$ maps any closed set to a closed set. We
define a set $\Delta\subset [0,1]$ by $$\Delta:=\{\delta\in [0,1]|1-\delta \Phi\ \text{is not onto} \}.$$ We claim that $\Delta$ is empty. If the claim holds then it obviously implies Lemma ~\ref{LM:surjective}.

We give an indirect proof of this claim. Suppose the claim is false. Then we define $\delta_0\in [0,1]$ by
$$\delta_0=\inf\{\delta|\ \delta\in \Delta\}.$$
\begin{lemma}\label{LM:contradict}
There exists a non-zero function $g_{0}\in L^{1}$ such that $$(1-\delta_{0}\Phi) g_{0}=0.$$
\end{lemma}
Obviously this contradicts Lemma ~\ref{LM:boundBelow}. \\
{\bf{Proof of Lemma ~\ref{LM:contradict}}}\\

We observe that $\delta_{0}\not=0,$ because the operator $\Phi$ is bounded.

Another observation is that the set $\Delta$ is closed: By Lemma ~\ref{LM:boundBelow}, the statement that $1-\delta \Phi$ is onto is equivalent to the statement that $1-\delta \Phi$ is invertible, and a classical result says that $\{\delta| 1-\delta \Phi\ \text{is invertible}\}$ is an open set.

Since $\Delta$ is closed, $\delta_0\in \Delta.$ Let $g_0\in L^{1}$ be a vector satisfying
\begin{equation}
g_{0}\not\in Range(1-\delta_{0}\Phi).
\end{equation}

Take a sequence $\{\epsilon_{n}\}_{n=0}^{\infty}\subset [0,\delta_0)$ satisfying $\displaystyle \lim_{n\rightarrow \infty}\epsilon_{n}=\delta_{0}$. By the definition, the maps $1-\epsilon_{n}\Phi$ are onto. This enables us to define a sequence of functions $\{g_{n}\}_{n=0}^{\infty}\subset L^{1}$ by setting $$g_{n}:=(1-\epsilon_{n}\Phi)^{-1}g_0.$$ The compactness of the operator $\Phi$ then implies that $$\|g_{n}\|_{L^{1}}\rightarrow \infty,\ \text{as}\ n\rightarrow \infty.$$

We set $\xi_{n}:=\frac{g_{n}}{\|g_{n}\|_{L^{1}}}.$ Then $$(1-\epsilon_{n}\Phi)\xi_{n}\rightarrow 0,\ \text{as}\ n\rightarrow \infty.$$ The fact that $\Phi$ is compact, together with arguments almost identical to those in proving ~\eqref{eq:noEigen}, then implies there exists a non-trivial function $g_{\infty}\in L^1$ such that $$(1-\delta_0\Phi)g_{\infty}=0.$$ This is Lemma ~\ref{LM:contradict}.
\begin{flushright}
$\square$
\end{flushright}
\subsection{Simplicity of the Eigenvalue $0$}
The following result has been used in the proof of Lemma ~\ref{LM:invertibility}. Denote the operator $e^{\frac{1}{2}|v|^2}K_0 e^{-\frac{1}{2}|v|^2}:\ L^{2}(\mathbb{R}^3)\rightarrow L^{2}(\mathbb{R}^{3})$ by $\tilde{K}_0.$ By the definition of $K_{0}$ in ~\eqref{eq:difK0} and the assumption on $r_0$ in ~\eqref{eq:NLBL1} it is easy to see that it is compact, self-adjoint and has a positive kernel.
\begin{lemma}\label{LM:lowestEi}
The linear self-adjoint unbounded operators $-\nu_{0}+\tilde{K}_0$, mapping $L^{2}(\mathbb{R}^3)$ to $L^{2}(\mathbb{R}^3),$ have the following properties
\begin{itemize}
\item[(A)] $0$ is a simple eigenvalue with eigenvector $e^{-\frac{1}{2}|v|^2};$
\item[(B)] there exists a constant $C>0$ such that if $g\in L^{2}(\mathbb{R}^3)$ is orthogonal to $ e^{-\frac{1}{2}|v|^2}$ then
    \begin{equation}
    \langle g,\ (-\nu_{0}+\tilde{K}_0)g\rangle \leq -C\|g\|_{2}^2.
    \end{equation}
\end{itemize}
\end{lemma}
\begin{proof}
The general idea in the proof is not new. It is similar to the proof of existence, uniqueness and positivity of ground states of Schr\"odinger operators; see \cite{Mcowen}.

Define $C_0$ as
\begin{equation}\label{eq:mini}
-C_0:=\inf \frac{\langle g,\ (\nu_{0}-\tilde{K}_0)g\rangle}{\langle g,\ g\rangle}.
\end{equation} The fact $(-\nu_0+\tilde{K}_0) e^{-\frac{1}{2}|v|^2}=0$ implies that $C_0\geq 0.$

By a series of transformations we find
\begin{equation}\label{eq:minimizer}
\begin{array}{lll}
0&=&\displaystyle\inf\frac{\langle g,\ (\nu_{0}+C_0-\tilde{K}_0)g\rangle}{\langle g,(\nu_{0}+C_0)g\rangle}\\
& &\\
&=&\displaystyle\inf\frac{\langle f,\ 1- (\nu_{0}+C_0)^{-\frac{1}{2}} \tilde{K}_0 (\nu_{0}+C_0)^{-\frac{1}{2}}f\rangle}{\langle f,\ f\rangle}.
\end{array}
\end{equation}
The key observation is that the operator $(\nu_{0}+C_0)^{-\frac{1}{2}} \tilde{K}_0 (\nu_{0}+C_0)^{-\frac{1}{2}}: \ L^{2}(\mathbb{R}^3)\rightarrow L^{2}(\mathbb{R}^3)$ is self-adjoint and compact, hence ~\eqref{eq:minimizer} has minimizers and they form a finite dimensional linear space. Suppose they are spanned by $\{\xi_{n}\}_{n=1}^{N} \subset L^{2}(\mathbb{R}^3)$, then each of them satisfies the equation
\begin{equation}\label{eq:secondEig}
[1-(\nu_{0}+C_0)^{-\frac{1}{2}} \tilde{K}_0 (\nu_{0}+C_0)^{-\frac{1}{2}}]\xi_n=0.
\end{equation} Moreover $$\displaystyle\inf_{f\perp \xi_{n},\ n=1,\cdots, N}\frac{\langle f,\ 1- (\nu_{0}+C_0)^{-\frac{1}{2}} \tilde{K}_0 (\nu_{0}+C_0)^{-\frac{1}{2}}f\rangle}{\langle f,\ f\rangle}>0,$$ hence, by defining $g=(\nu_0+C_0)^{-\frac{1}{2}}f,$ we obtain
\begin{equation}\label{eq:lowerBound}
\displaystyle\inf_{g\perp (\nu+C_{0})^{\frac{1}{2}}\xi_{n},\ n=1,\cdots,N}\frac{\langle g,\ (\nu_{0}+C_0-\tilde{K}_0)g\rangle}{\langle g,(\nu_{0}+C_0)g\rangle}>0.
\end{equation}

In the next we prove the minimizer is unique. Since the operator $(\nu_{0}+C_0)^{-\frac{1}{2}} \mathcal{K}_0 (\nu_{0}+C_0)^{-\frac{1}{2}}$ is compact and its integral kernel is strictly positive, we find that $$\langle |f|,\ 1- (\nu_{0}+C_0)^{-\frac{1}{2}} \tilde{K}_0 (\nu_{0}+C_0)^{-\frac{1}{2}}|f|\rangle\leq \langle f,\ 1- (\nu_{0}+C_0)^{-\frac{1}{2}} \tilde{K}_0 (\nu_{0}+C_0)^{-\frac{1}{2}}f\rangle.$$ Noticing that $
\langle f,\ f\rangle=\langle |f|,\ |f|\rangle,$ we see that if $\xi$ is a minimizer, so is $|\xi |.$ Hence $|\xi|$ is a one of the solutions to ~\eqref{eq:secondEig}. This, together with the fact the integral kernel of $(\nu_{0}+C_0)^{-\frac{1}{2}} \tilde{K}_0 (\nu_{0}+C_0)^{-\frac{1}{2}}$ is strictly positive, implies that $|\xi|$ is strictly positive and $\xi=|\xi|$ or $-|\xi|.$
This in turn implies that the minimizer is unique and positive, up to a sign. And, moreover, the nonnegative function $\eta:=(\nu_0+C_0)^{-\frac{1}{2}}\xi$ satisfies the equation $$(-\nu_0+\tilde{K}_0-C_0)\eta=0,$$ i.e., $\eta$ is an eigenvector with eigenvalue $C_0.$

Furthermore, the strictly positive function $e^{-\frac{1}{2}|v|^2}$ is an eigenvector of $-\nu_0+\tilde{K}_0$ with eigenvalue zero. It is not orthogonal to the minimizer $\eta$. This forces the unique minimizer $\eta$ to be parallel to $e^{-\frac{1}{2}|v|^2}$ and, moreover, $C_0=0.$ This is statement (A).

To verify Statement (B), we derive from ~\eqref{eq:lowerBound} that
$$\inf_{g\perp \nu e^{-\frac{1}{2}|v|^2}}\frac{\langle g, (\nu_0-\tilde{K}_0)g\rangle}{\langle g,\ g\rangle}>0.$$
This together with the fact $\nu_0 e^{-\frac{1}{2}|v|^2}\not\perp e^{-\frac{1}{2}|v|^2}$ and the min-max principle implies statement B.

These results complete the proof of the lemma.
\end{proof}

\def\cprime{$'$} \def\cprime{$'$} \def\cprime{$'$} \def\cprime{$'$}


\begin{thebibliography}{10}

\bibitem{Ark1988}
L.~Arkeryd.
\newblock Stability in {$L^1$} for the spatially homogeneous {B}oltzmann
  equation.
\newblock {\em Arch. Rational Mech. Anal.}, 103(2):151--167, 1988.

\bibitem{MR900501}
L.~Arkeryd, R.~Esposito, and M.~Pulvirenti.
\newblock The {B}oltzmann equation for weakly inhomogeneous data.
\newblock {\em Comm. Math. Phys.}, 111(3):393--407, 1987.

\bibitem{Boby1997}
A.~V. Bobylev.
\newblock Moment inequalities for the {B}oltzmann equation and applications to
  spatially homogeneous problems.
\newblock {\em J. Statist. Phys.}, 88(5-6):1183--1214, 1997.

\bibitem{BobyGamPan04}
A.~V. Bobylev, I.~M. Gamba, and V.~A. Panferov.
\newblock Moment inequalities and high-energy tails for {B}oltzmann equations
  with inelastic interactions.
\newblock {\em J. Statist. Phys.}, 116(5-6):1651--1682, 2004.

\bibitem{Bodmer1973}
R.~Bodmer.
\newblock Zur {B}oltzmanngleichung.
\newblock {\em Comm. Math. Phys.}, 30:303--334, 1973.

\bibitem{Carleman1960}
T.~Carleman.
\newblock {\em Matematicheskie zadachi kineticheskoi teorii gazov}.
\newblock Translated from the French by V.-K. I. Karabegov; edited by N. N.
  Bogolyubov. Biblioteka Sbornika ``Matematika''. Izdat. Inostr. Lit., Moscow,
  1960.

\bibitem{CerIllPulv}
C.~Cercignani, R.~Illner, and M.~Pulvirenti.
\newblock {\em The mathematical theory of dilute gases}, volume 106 of {\em
  Applied Mathematical Sciences}.
\newblock Springer-Verlag, New York, 1994.

\bibitem{CourHil1989}
R.~Courant and D.~Hilbert.
\newblock {\em Methods of mathematical physics. {V}ol. {II}}.
\newblock Wiley Classics Library. John Wiley \& Sons Inc., New York, 1989.
\newblock Partial differential equations, Reprint of the 1962 original, A
  Wiley-Interscience Publication.

\bibitem{Davies1976}
E.~B. Davies.
\newblock {\em Quantum theory of open systems}.
\newblock Academic Press [Harcourt Brace Jovanovich Publishers], London, 1976.

\bibitem{MR1233644}
L.~Desvillettes.
\newblock Some applications of the method of moments for the homogeneous
  {B}oltzmann and {K}ac equations.
\newblock {\em Arch. Rational Mech. Anal.}, 123(4):387--404, 1993.

\bibitem{MR2116276}
L.~Desvillettes and C.~Villani.
\newblock On the trend to global equilibrium for spatially inhomogeneous
  kinetic systems: the {B}oltzmann equation.
\newblock {\em Invent. Math.}, 159(2):245--316, 2005.

\bibitem{DiPernaLion1989}
R.~J. DiPerna and P.-L. Lions.
\newblock On the {C}auchy problem for {B}oltzmann equations: global existence
  and weak stability.
\newblock {\em Ann. of Math. (2)}, 130(2):321--366, 1989.

\bibitem{ErdhosYau}
L.~Erd{\H{o}}s and H.-T. Yau.
\newblock Linear {B}oltzmann equation as the weak coupling limit of a random
  {S}chr\"odinger equation.
\newblock {\em Comm. Pure Appl. Math.}, 53(6):667--735, 2000.

\bibitem{Glassey1996}
R.~T. Glassey.
\newblock {\em The {C}auchy problem in kinetic theory}.
\newblock Society for Industrial and Applied Mathematics (SIAM), Philadelphia,
  PA, 1996.

\bibitem{MR0156656}
H.~Grad.
\newblock Asymptotic theory of the {B}oltzmann equation. {II}.
\newblock In {\em Rarefied {G}as {D}ynamics ({P}roc. 3rd {I}nternat. {S}ympos.,
  {P}alais de l'{UNESCO}, {P}aris, 1962), {V}ol. {I}}, pages 26--59. Academic
  Press, New York, 1963.

\bibitem{MR2629879}
P.~T. Gressman and R.~M. Strain.
\newblock Global classical solutions of the {B}oltzmann equation with
  long-range interactions.
\newblock {\em Proc. Natl. Acad. Sci. USA}, 107(13):5744--5749, 2010.

\bibitem{Mouhot2010}
M.~Gualdani, S.~Mischler and C.~Mouhot.
\newblock Factorization for non-symmetric operators and exponential H-theorem.
\newblock {\em arxiv.org/abs/1006.5523, 2010}.

\bibitem{YGuo2002}
Y.~Guo.
\newblock The {V}lasov-{P}oisson-{B}oltzmann system near {M}axwellians.
\newblock {\em Comm. Pure Appl. Math.}, 55(9):1104--1135, 2002.

\bibitem{YGuo2003}
Y.~Guo.
\newblock The {V}lasov-{M}axwell-{B}oltzmann system near {M}axwellians.
\newblock {\em Invent. Math.}, 153(3):593--630, 2003.

\bibitem{Lanford1975}
O.~E. Lanford, III.
\newblock Time evolution of large classical systems.
\newblock In {\em Dynamical systems, theory and applications ({R}econtres,
  {B}attelle {R}es. {I}nst., {S}eattle, {W}ash., 1974)}, pages 1--111. Lecture
  Notes in Phys., Vol. 38. Springer, Berlin, 1975.

\bibitem{LuX1998}
X.~Lu.
\newblock A direct method for the regularity of the gain term in the
  {B}oltzmann equation.
\newblock {\em J. Math. Anal. Appl.}, 228(2):409--435, 1998.

\bibitem{Mcowen}
R.~McOwen.
\newblock {\em Partial Differential Equations: Methods and Applications}.
\newblock Prentice Hall, New Jersey, 2002.

\bibitem{MisWenn1999}
S.~Mischler and B.~Wennberg.
\newblock On the spatially homogeneous {B}oltzmann equation.
\newblock {\em Ann. Inst. H. Poincar\'e Anal. Non Lin\'eaire}, 16(4):467--501,
  1999.

\bibitem{MR2148153}
M.~Mokhtar-Kharroubi and M.~Sbihi.
\newblock Critical spectrum and spectral mapping theorems in transport theory.
\newblock {\em Semigroup Forum}, 70(3):406--435, 2005.

\bibitem{MR2216092}
M.~Mokhtar-Kharroubi and M.~Sbihi.
\newblock Spectral mapping theorems for neutron transport, {$L^1$}-theory.
\newblock {\em Semigroup Forum}, 72(2):249--282, 2006.

\bibitem{Mouhot2006}
C.~Mouhot.
\newblock Rate of convergence to equilibrium for the spatially homogeneous
  {B}oltzmann equation with hard potentials.
\newblock {\em Comm. Math. Phys.}, 261(3):629--672, 2006.

\bibitem{MouVi}
C.~Mouhot and C.~Villani.
\newblock Regularity theory for the spatially homogeneous {B}oltzmann equation
  with cut-off.
\newblock {\em Arch. Ration. Mech. Anal.}, 173(2):169--212, 2004.

\bibitem{RSI}
M.~Reed and B.~Simon.
\newblock {\em Methods of modern mathematical physics. {I}. {F}unctional
  analysis}.
\newblock Academic Press, New York, 1972.

\bibitem{RezVillani2008}
F.~Rezakhanlou and C.~Villani.
\newblock {\em Entropy methods for the {B}oltzmann equation}, volume 1916 of
  {\em Lecture Notes in Mathematics}.
\newblock Springer, Berlin, 2008.
\newblock Lectures from a Special Semester on Hydrodynamic Limits held at the
  Universit{\'e} de Paris VI, Paris, 2001, Edited by Fran{\c{c}}ois Golse and
  Stefano Olla.

\bibitem{UKai1974}
S.~Ukai.
\newblock On the existence of global solutions of mixed problem for non-linear
  {B}oltzmann equation.
\newblock {\em Proc. Japan Acad.}, 50:179--184, 1974.

\bibitem{Wenn1993}
B.~Wennberg.
\newblock Stability and exponential convergence in {$L^p$} for the spatially
  homogeneous {B}oltzmann equation.
\newblock {\em Nonlinear Anal.}, 20(8):935--964, 1993.

\bibitem{MR1264851}
B.~Wennberg.
\newblock On moments and uniqueness for solutions to the space homogeneous
  {B}oltzmann equation.
\newblock {\em Transport Theory Statist. Phys.}, 23(4):533--539, 1994.

\bibitem{Wenn1994}
B.~Wennberg.
\newblock Regularity in the {B}oltzmann equation and the {R}adon transform.
\newblock {\em Comm. Partial Differential Equations}, 19(11-12):2057--2074,
  1994.

\end{thebibliography}
\end{document}